\documentclass[12pt,final]{amsart}

\usepackage{amsmath,amssymb,amsthm,color,enumerate}
\usepackage[dvipsnames]{xcolor}
\usepackage{fullpage}
\usepackage{mathtools}
\usepackage{amsfonts}
\usepackage{tikz-cd}
\usepackage{graphicx}
\usepackage{kantlipsum}

\usepackage{datetime2}
\usepackage{comment}

\usepackage[normalem]{ulem} 
\usepackage{apptools}

\usepackage{tikz}
\usepackage{todonotes}
\usetikzlibrary{positioning}

\usepackage{amsmath, amsthm}
\usepackage{amsfonts,amssymb}

\newcommand{\bc}{\begin{center}}
\newcommand{\ec}{\end{center}}
\newcommand{\bt}{\begin{tabular}}
\newcommand{\et}{\end{tabular}} 
\newcommand{\bea}{\begin{eqnarray}}
\newcommand{\eea}{\end{eqnarray}}
\newcommand{\bean}{\begin{eqnarray*}}
\newcommand{\eean}{\end{eqnarray*}}

\newcommand{\ba}{\begin{array}}
\newcommand{\ea}{\end{array}}

\def\be{\begin{eqnarray}}
\def\ee{\end{eqnarray}}
\def\ben{\begin{eqnarray*}}
\def\een{\end{eqnarray*}}

\newcommand{\RL}{{\mathbb R}}
\newcommand{\NN}{{\mathbb N}}

\def\elabel#1{\label{e:#1}}

\def\sq{$\Box$}

\def\qed{\ifmmode\sq\else{\unskip\nobreak\hfil
\penalty50\hskip1em\null\nobreak\hfil\sq
\parfillskip=0pt\finalhyphendemerits=0\endgraf}\fi\par\medbreak}

\newsavebox{\junk}
\savebox{\junk}[1.6mm]{\hbox{$|\!|\!|$}}

\def\det{{\mathop{\rm det}}}

\def\ind{\field{I}}

\def\til={{\widetilde =}}

 \def\eq#1/{(\ref{#1})}

\def\eq#1/{(\ref{e:#1})}

\newcommand{\beqn}[1]{\notes{#1}%
\begin{eqnarray} \elabel{#1}}

\newcommand{\eeqn}{\end{eqnarray} }

\newcommand{\beq}[1]{\notes{#1}%
\begin{equation}\elabel{#1}}

\newcommand{\eeq}{\end{equation}} 

\def\bdes{\begin{description}}
\def\edes{\end{description}}

\def\notes#1{}

\def\E{\mathbb{E}}
\def\C{{\bf C}}

\def\Var{{\rm Var}}

\def\phi{\varphi}

\def\bee{\begin{eqnarray*}}
\def\ene{\end{eqnarray*}}

\newcommand{\R}{\mathbb{R}}
\newcommand{\N}{\mathbb{N}}
\newcommand{\Z}{\mathbb{Z}}

\usepackage{amsfonts}
\usepackage{amssymb}
\usepackage{amsmath}
\usepackage{amsthm}
\usepackage{amsbsy}
\usepackage{amssymb} 
\usepackage{verbatim}
\usepackage{bm} 
\usepackage{paralist} 
 \usepackage{color}
 \usepackage{mathrsfs}
 \usepackage{graphicx}
\usepackage{hyperref}

\def\Var{\mbox{Var}}

\def\ind{{\mathbf 1}}

\def\to{\rightarrow}

\def\mb{\mbox}

\def\l{\left}
\def\r{\right}
\def\<{\langle}
\def\>{\rangle}

\newcommand\mnote[1]{} 
\newcommand\bes{\begin{eqnarray*}}
\newcommand\ees{\end{eqnarray*}}
\newcommand\besn{\begin{eqnarray}}
\newcommand\eesn{\end{eqnarray}}

\def\bthm{\begin{theorem}}
\def\ethm{\end{theorem}}
\def\bdefn{\begin{definition}}
\def\edefn{\end{definition}}
\newcommand{\benu}{\begin{enumerate}\setlength\itemsep{6pt}}
\newcommand{\beit}{\begin{itemize}\setlength\itemsep{3pt}}
\def\eenu{\end{enumerate}}
\def\eeit{\end{itemize}}
\def\beds{\begin{description}}
\def\eeds{\end{description}}
\def\bepr{\begin{problem}}
\def\eepr{\end{problem}}

\def\bprf{\begin{proof}}
\def\eprf{\end{proof}}
\def\berk{\begin{remark}}
\def\eerk{\end{remark}}
\def\bex{\begin{exercise}}
\def\eex{\end{exercise}}
\def\beg{\begin{example}}
\def\eeg{\end{example}}

\def\suchthat{{\; : \;}}
\renewcommand{\qed}{\hfill\text{$\blacksquare$}}

\def\HH{{\mathcal H}}

\def\PP{{\mathcal P}}

\def\C{\mathbb{C}}
\def\N{\mathbb{N}}
\def\R{\mathbb{R}}

\def\Z{\mathbb{Z}}
\def\PP{\mathbb{P}}

\newcommand{\sm}{{\raise0.3ex\hbox{$\scriptstyle \setminus$}}}

\renewcommand\phi{\varphi}

\theoremstyle{plain} 
    \newtheorem{theorem}{Theorem}
    \newtheorem{lemma}{Lemma}
    
    \newtheorem{corollary}{Corollary}
    
    \newtheorem{conjecture}{Conjecture}

\theoremstyle{definition} 
    \newtheorem{definition}{Definition}
    
    \newtheorem{result}[theorem]{Result}

    \newtheorem{exercise}[theorem]{Exercise}
    \newtheorem{problem}[theorem]{Problem}
        \newtheorem{remark}{Remark}
    \newtheorem{example}[theorem]{Example}

\usepackage{hyperref}
\hypersetup{
    colorlinks=true,
    linkcolor=blue,
    filecolor=magenta,      
    urlcolor=cyan
    }

\usepackage{palatino}

\title{Log-concavity in one-dimensional Coulomb gases and related ensembles}

\author{Jnaneshwar Baslingker, Manjunath Krishnapur, Mokshay Madiman }
\date{\today}

\address{Department of Mathematics\\
        University of Toronto\\
        Toronto, ON\\
        Canada}

\email{j.baslingker@utoronto.ca}

\address{Department of Mathematics\\
        Indian Institute of Science\\
        Bangalore 560012, India}

\email{manju@iisc.ac.in}

\address{University of Delaware\\
Department of Mathematical Sciences\\ 501 Ewing Hall
Newark, DE 19716, USA}

\email{madiman@udel.edu}

\thanks{M.K. is partly supported by the DST FIST program - 2021 [TPN - 700661]. We acknowledge the support of the International Centre for Theoretical Sciences (ICTS) as this work was initiated when the authors participated in the program  {\em Topics in High Dimensional Probability} (code: ICTS/thdp2023/1).}

\begin{document}
\begin{abstract}
We prove log-concavity of the lengths of the top rows of Young diagrams under Poissonized Plancherel measure. This is the first known positive result towards a 2008 conjecture  of Chen~\cite{Che08} that the length of the top row of a Young diagram under the Plancherel measure is log-concave. This is done by showing that the ordered elements of several discrete ensembles
have log-concave distributions. In particular, we show the log-concavity of passage times in last passage percolation with geometric weights,  using their connection to Meixner ensembles.

In the continuous setting, distributions of the maximal elements of  beta ensembles with convex potentials on the real line are shown to be log-concave. As a result, log-concavity of the Tracy-Widom distributions for all parameters $\beta>0$ follows, confirming a folklore conjecture that was partially proved by Deift for $\beta=2$. Furthermore, we also obtain log-concavity and positive association for the joint distribution of the $k$ smallest eigenvalues of the stochastic Airy operator. Our methods also show the log-concavity of the Airy-2 process and the Airy distribution. A log-concave distribution with full-dimensional support must have density, a fact that was apparently not known for some of these examples.
\end{abstract}
\date{}
\maketitle

\section{Introduction}
A Radon measure $\mu$ on $\R^n$ is said to be log-concave if 
\[
\mu(s A+(1-s)B)\ge \mu(A)^{s}\mu(B)^{1-s}
\]
for all Borel sets $A,B$ and for all $0\leq s\leq 1$. Here $A+B=\{a+b\; : \; a\in A, \ b\in B\}$ is the Minkowski sum. It is a well-known result of Borell (see Theorem~2.7 of \cite{SW14}) that if $\mu$ is not supported in any $n-1$ dimensional affine subspace, then $\mu$ is absolutely continuous with respect to Lebesgue measure and has a density function (i.e., Radon-Nikodym derivative) that is log-concave. Recall that a non-negative function $f$ defined on $\mathbb{R}^n$ is said to be log-concave if 
\begin{align*}
    f(s x+(1-s)y)\geq f(x)^{s}f(y)^{1-s},
\end{align*}
for each $x,y\in\mathbb{R}^n$ and $0\leq s \leq 1$. In the discrete setting, a sequence $\{a_k\}_{k\in \mathbb {Z}}$ of non-negative numbers is said to log-concave if $a_{k}^2\geq a_{k-1}a_{k+1}$ for all $k$ and there are no internal zeros. There is no universally accepted notion of log-concavity  on $\Z^n$. We propose the following definition, for reasons that we discuss in Section~\ref{sec: disc mult lc}.
\begin{definition}\label{def: disc lc}
    A non-negative function $f$ on $\Z^n$ is said to be discrete log-concave if 
    \[
    f(x)f(y)\le f\l(\l\lfloor \frac{x+y}{2}\r\rfloor\r)f\l(\l\lceil \frac{x+y}{2}\r\rceil\r)
    \]
    for all $x,y\in \Z^n$, where the sum, floor and ceiling functions are applied co-ordinatewise.
\end{definition}
This agrees with the definition of log-concavity for $n=1$ (note that internal zeros are ruled out by this definition). Most importantly, with this definition, log-concavity is preserved under (co-ordinate) marginalization. This is explained in Section~\ref{sec: disc mult lc}.

 A random variable or its probability distribution is said to be log-concave if it has a log-concave density function (on $\mathbb{R}^n$) or if it has a discrete log-concave mass function (on $\Z^n$).

Log-concave distributions and several properties related to it play an important role in several areas of mathematics and therefore have been extensively studied. Applications of log-concavity arise in combinatorics, algebra and computer science, as reviewed by Stanley \cite{Sta99:book} and Brenti \cite{Bre89:book}. In probability, it is related to the notion of negative association of random variables \cite{BBL09}, and is also useful in statistics (see, e.g.,  \cite{JP83, SW14}). 
Log-concave distributions also arise very organically in convex geometry and geometric functional analysis (see, e.g., \cite{BM12:jfa, MMX17:0}). Several functional inequalities that hold for Gaussian distributions also hold for appropriate subclasses of log-concave distributions on $\RL^n$ (see, e.g., \cite{Bob99, BBCG08, BM11:aop}). Thus, knowing that a distribution is log-concave gives much  information about the distribution. In this article, the ordered elements in several one-dimensional Coulomb gas ensembles arising in probability and mathematical physics are shown to have log-concave distributions. 

Many new and exotic probability distributions have arisen in random matrix theory and related areas in the last few decades. Usually these distributions are described as weak limits of random variables in some discrete or continuous finite systems that are growing in size. Even when there is an explicit formula for the density of the limiting distribution, it is often too complicated. Further, in the discrete setting, log-concavity of various sequences has attracted much recent attention (see \cite{Mas72, AHK18, HSW22, ALOV24}), but there are many other conjectures as yet unresolved. Our main contributions in this paper are two-fold:
\begin{enumerate}

 \item We show the log-concavity of many of these exotic distributions. Examples include $\beta$ versions of  Tracy-Widom distributions (including the classical cases of $\beta=1,2,4$, where the result is already new), finite dimensional distributions of the Airy-2 process,  passage time distributions in integrable models of last passage percolation, and the Airy distribution. This adds to our knowledge of these important distributions. Even in the important case of the $\beta=2$ Tracy-Widom distribution, log-concavity was only partially known (see \cite{BLS17}).
 
 \item From the log-concavity of passage times in last passage percolation with geometric weights, we derive the log-concavity of the Poissonized length of the longest increasing subsequence of a uniform random permutation. The motivation for this result comes from a conjecture of Chen~\cite{Che08}, to the effect that the distribution of the longest increasing subsequence of a uniform random permutation of $\{1,\ldots ,n\}$, is itself log-concave. This conjecture has attracted the attention of combinatorialists, see for example B\'{o}na, Lackner and Sagan~\cite{BLS17}. As far as we know, ours is the first positive result in this direction.

\end{enumerate}

After the first version of this paper was publicly posted, several works have built on our methods and results, including \cite{XW24, DLM24}.


\section{Results in the discrete setting}
\subsection{Chen's conjecture}

Let $\mathcal{S}_n$ be the symmetric group on $[n]$, i.e., the set of all permutations of $[n]=$ $\{1,2,\dots,n\}$. Let $\ell_n(\sigma)$ denote the length of the longest increasing subsequence of the permutation $\sigma\in\mathcal{S}_n$. For example, if $\sigma=42135$, then $\ell_5(\sigma)=3$ as $2,3,5$ is an increasing subsequence of length $3$. 
The asymptotics of $\ell_n(\sigma)$ for a uniformly chosen random permutation is very well understood. The work of Logan and Shepp \cite{LS77}, Vershik and Kerov \cite{VK77,VK85} shows that $\frac{\ell_n(\sigma)}{\sqrt{n}}\rightarrow 2$ in probability and expectation as $n\rightarrow \infty$.
Baik, Deift and Johansson \cite{BDJ98} prove that $\ell_n(\sigma)$ after appropriate scaling and centering converges in distribution to $TW_2$. 
Romik's book \cite{Romik15} gives a wide-ranging view of many aspects of longest increasing subsequences.

Define 
\begin{align*}
    L_{n,k}=\{\sigma\in\mathcal{S}_n:\ell_n(\sigma)=k\}\quad \mbox{and}\quad \ell_{n,k}=|L_{n,k}|.
\end{align*}
Chen \cite{Che08} made the following conjecture. See \cite{BLS17} for more about the conjecture.

\begin{conjecture}[Chen]\label{conj: Chen conjecture}
For any fixed $n$, the sequence $\ell_{n,1}, \ell_{n,2}, \ldots, \ell_{n,n}$ is log-concave.
\end{conjecture}
In other words, the conjecture states that the distribution of $\ell_n(\sigma)$, where $\sigma$ is uniformly chosen random permutation, is log-concave. B\'ona-Lackner-Sagan~\cite{BLS17}  made a similar conjecture when $\sigma$ is a uniformly chosen random involution. We consider both problems in the setting of Young diagrams.

Let $\Lambda_n$ denote the set of integer partitions of $n$, also identified with Young diagrams having $n$ boxes. Let $\Lambda=\cup_{n=0}^{\infty}\Lambda_n$. Elements of $\Lambda_n$ are of the form $\lambda=(\lambda_1,\lambda_2,...,\lambda_{\ell},0,0...)$ where $\lambda_1\geq\lambda_2\geq\dots\geq\lambda_{\ell}\geq 1$ are positive integers and $\sum_{i}\lambda_i=n$. We write   $\lambda\vdash n$ to mean $\lambda\in \Lambda_n$.  
Given a partition $\lambda\vdash n$, let $d_{\lambda}$ denote the number of standard Young tableaux of shape $\lambda$. 

We consider the $\beta$-Plancherel measure (any real $\beta>0$) $\mu_n^{(\beta)}$ on $\Lambda_n$ defined by,
\begin{align}\label{eq: beta plancherel measure}
    \mu_n^{(\beta)}(\lambda):=\frac{d_{\lambda}^{\beta}}{\sum\limits_{\tau\vdash n}d_{\tau}^{\beta}},\quad \lambda\in\Lambda_n.
\end{align}

$\beta$-Plancherel measures have been studied previously in \cite{BaikRains01, Regev81}. For $\beta=2$, this is the Plancherel measure which arises in representation theory. The Plancherel measure on partitions $\Lambda_n$ arises naturally and is well studied in representation–theoretic, combinatorial, and probabilistic problems \cite{VK77, LS77, BOO99}. 
By the Robinson-Schensted correspondence \cite{Sta99:book},  Conjecture \ref{conj: Chen conjecture} is equivalent to 
\begin{align}\label{eq: RSK chen conjecture}
\mu_n^{(2)}(\lambda_1=k-1)\mu_n^{(2)}(\lambda_1=k+1)\leq (\mu_n^{(2)}(\lambda_1=k))^2,    
\end{align}
which is the log-concavity of the distribution of length of first row under the Plancherel measure $\mu_{n}^{(2)}$ on $\Lambda_n$. The corresponding inequality for $\beta=1$ is equivalent to the B\'ona-Lackner-Sagan conjecture on involutions \cite[Conjecture $1.2$]{BLS17}.

One of our main results is that the distribution of $\lambda_1$ is  log-concave for a family of mixtures of $\mu_{n}^{(\beta)}$. For $\beta=2$, the mixture is a Poissonization, which has been studied before~\cite{BOO99, BDJ98}. In fact, the limiting distribution of fluctuations of $\ell_n(\sigma)$ is derived in \cite{BDJ98} using the determinantal structure of Poissonized Plancherel measure on $\Lambda$.

For the rest of the article, we assume $\NN=\{0,1,2,\dots\}$. For parameters $\alpha,\beta>0$, consider the family of probability measures ${\nu_{\alpha,\beta}}$ on $\NN$ defined such that,
\begin{align}\label{eq: mixture distribution}
    {\nu_{\alpha,\beta}(k)}=\frac{1}{Z_{\alpha,\beta}}\alpha^k{\sum\limits_{\lambda\vdash k}(d_{\lambda}/k!)^{\beta}}.
\end{align}

 That $Z_{\alpha,\beta}$ is finite follows from  $\max\limits_{\lambda\vdash k}d_{\lambda}\leq \sqrt{k!}$ (easy consequence of the identity $\sum_{\lambda\vdash k}d_{\lambda}^2=k!$) and $|\Lambda_k|\le e^{C\sqrt{k}}$ (see pp. 316-318 of \cite{apostol}). We define the mixture of $\mu_{n}^{\beta}$, denoted as ${M^{(\alpha,\beta)}}$, to be the probability measure on $\Lambda$, where $X\sim\nu_{\alpha,\beta}$ and sample $\lambda\in\Lambda_{X}$ under $\mu_{X}^{(\beta)}$. For $\beta=2$, note that $\nu_{\alpha,2}$ is the Poisson($\alpha$) distribution and hence $M^{(\alpha,2)}$ is the Poissonized Plancherel measure with $\alpha$ being the Poisson parameter. Our first main result is the following. 

\begin{theorem}\label{thm: Poisson Plancherel log-concave}
      For any $i\geq 1$ and $\alpha,\beta>0$, the distribution of $\lambda_i$ under the probability measure $M^{(\alpha,\beta)}$ is log-concave. 
 \end{theorem}

 For $\beta=2$ and $\beta=1$ in Theorem \ref{thm: Poisson Plancherel log-concave}, we obtain Poissonized version of Chen's Conjecture and a certain mixture version of B\'{o}na-Lackner-Sagan's conjecture respectively. This neither implies Chen's conjecture nor is implied by it. However, when $\alpha=n$, the measure $\nu_{\alpha,2}$ has mean $n$ and standard deviation $\sqrt{n}$, therefore $M^{(n,2)}$ is quite close to $\mu_n^{(2)}$. In that sense, Theorem~\ref{thm: Poisson Plancherel log-concave} supports Chen's conjecture and even suggests that it may  strengthened to log-concavity of $\lambda_i$ for any $i$, under $\mu_n^{(\beta)}$ for general $\beta>0$.

It was remarked in \cite{BLS17} that proving log-concavity of $TW_2$ distribution (which is the limiting distribution of fluctuations of $\ell_n(\sigma)$) could be a possible approach to prove Conjecture \ref{conj: Chen conjecture}. What is definitely true is that for Conjecture \ref{conj: Chen conjecture} to be true, $TW_2$ has to be log-concave. 

\begin{lemma}\label{lem: Discrete to continuous}
    Let $\{X_n:n\in\NN\}$ be $\Z$-valued log-concave random variables and $\frac{X_n-a_n}{b_n}\overset{d}{\rightarrow} Y$, where $Y$ is a random variable with density function $f$ and $a_n,b_n$ are some sequences. Then $f$ is log-concave.
\end{lemma}
By the above lemma, Theorem $1$ of \cite{BDJ98} and Theorem~\ref{thm: Poisson Plancherel log-concave}, it follows that $TW_2$ is log-concave.  In this paper, we give multiple proofs that $TW_2$ and its $\beta$ generalizations are log-concave, the proof of Corollary \ref{cor: Stoch Airy operator} being the simplest one.  Although Tracy-Widom distributions are widely studied, the log-concavity property does not seem to have been observed before. In fact, in \cite{BLS17}, only a partial proof (due to P. Deift) is given, showing the log-concavity of $TW_2$ on the positive half line.

The reason that these specific mixtures are amenable to study is that they are related to the Meixner ensemble (defined below). In particular, Theorem~\ref{thm: Poisson Plancherel log-concave} follows from the log-concavity of individual particles in the Meixner ensemble. The Meixner ensemble falls inside two larger classes of particle systems on $\mathbb Z$, namely, discrete ensembles that resemble Coulomb gases and Schur measures. In both of these classes, we show log-concavity of marginals.

As additional evidence to Conjecture \ref{conj: Chen conjecture}, we prove the following partial result.

\begin{theorem}\label{thm: Chen conjecture partial proof}
    Fix $j\in\NN$. Then $\exists N=N(j)$ such that, $\forall n\geq N$ and $k\in\{n-j,\dots,n\}$,
    \begin{align}\label{eq: Chen conj result}
        \mu_n^{(2)}(\lambda_1=k-1)\mu_n^{(2)}(\lambda_1=k+1)\leq (\mu_n^{(2)}(\lambda_1=k))^2.
    \end{align}
\end{theorem}

 In order to prove Conjecture \ref{conj: Chen conjecture}, we cannot use Theorem \ref{thm: Poisson Plancherel log-concave} as preservation of log-concavity under depoissonization or Poissonization is not guaranteed. In this direction, we provide sufficient conditions under which Poissonization of a sequence of probability measures is log-concave.

Let $\mu_0,\mu_1,\dots$ be a sequence of probability distributions on $\NN$ and let $Y\sim \mu_X$ where $X\sim \mbox{Poisson}(\lambda)$ for some $\lambda>0$. Then we say $Y$ is Poissonization of the sequence $\mu_0,\mu_1,\dots$. A natural question is under what conditions does the random variable $Y$ have log-concave distribution. We prove the following theorem which provides a sufficient condition for $Y$ to have log-concave distribution. 
\begin{theorem}\label{thm: Poissonization}
    Let $\mu_0,\mu_1,\mu_2\dots$ be such that $\forall i,j\in\NN\cup \{0\}$ and $k\geq 2$,
    \begin{align}\label{eq: poisson condition}
        \frac{\mu_i(k-1)}{i!}\frac{\mu_j(k+1)}{j!}\leq \frac{\mu_{\left\lfloor\frac{i+j}{2}\right\rfloor}(k)}{\left\lfloor\frac{i+j}{2}\right\rfloor!}\frac{\mu_{\left\lceil\frac{i+j}{2}\right\rceil}(k)}{\left\lceil\frac{i+j}{2}\right\rceil!}.
    \end{align}

Then $Y\sim\mu_{X}$ has log-concave distribution where $X\sim \mbox{Poisson}(\lambda)$.
\end{theorem}

\subsection{Log-concavity in discrete ensembles}
 For $w_i,Q_{i,j}:\Z\to \R_+$ with $1\leq i<j\leq n$, define the   probability measure on $\overrightarrow{\Z}^n=\{h\in\Z^n:  h_1<h_2<\dots<h_n\}$ 
\begin{align}\label{eq: DOPE}
  \PP_{n,w,Q}(h)=  \frac{1}{Z}\prod\limits_{1\leq i<j\leq n}Q_{i,j}\left(h_j-h_i\right)\prod\limits_{j=1}^{n}w_j(h_j), \ h\in\overrightarrow{\Z}^n
\end{align}
where $Z=Z_{n,w,Q}$ is a normalisation constant. We set  $\PP_{n,w,Q}(h)=0$ for $h\in \Z^n\setminus \overrightarrow{\Z}^n$. Of course, appropriate conditions are imposed on $Q_{i,j}$ and $w_i$ for $\PP_{n,w,Q}(h)$ to exist. This can be thought of as a discrete analogue of Coulomb gas. Although most of the important examples of discrete ensembles have $Q_{i,j}=Q$ and $w_i=w$ for all $1\leq i<j\leq n$, we consider the general definition given in \eqref{eq: DOPE} in order to include examples like \eqref{eq:anotherformofdiscretebetagas}. For $Q_{i,j}(x)=Q(x)=x^2$ and $w_i(x)=w(x)$ we will refer to \eqref{eq: DOPE} as a discrete orthogonal polynomial ensemble, following Johansson~\cite{KJ02}.
 Our second main result is the following. Recall the notion of discrete log-concavity from Definition~\ref{def: disc lc}.

\begin{theorem}\label{thm:gen weight}
 Assume that $w_i(x),Q_{i,j}(x)$ are log-concave sequences on $\Z$ for all $1\leq i<j\leq n$, that is 
\begin{align}\label{eq:log conc org}
    w_i(k-1)w_i(k+1)\leq w_i(k)^2, \\
    Q_{i,j}(k-1)Q_{i,j}(k+1)\leq Q_{i,j}(k)^2,
\end{align}
for all $k\in \Z$. Then, $\PP_{n,w,Q}$ is discrete log-concave on $\Z^n$. Further, for any  $i\in[n]$, the distribution of $h_i$ under $\PP_{n,w,Q}$ is log-concave, that is
\begin{align}
    \PP_{n,w,Q}(h_i=k-1)\PP_{n,w,Q}(h_i=k+1)\leq \PP_{n,w,Q}(h_i=k)^2. \label{eq:pmf log conc}
\end{align}
\end{theorem}

\begin{remark}\label{rem: ulc}
A sequence $\{a_n\}_{n\in\NN}$ is said to be ultra-log-concave (of infinite order) if $\{n!a_n\}_{n\in\NN}$ is log-concave (cf., \cite{liggett1997ultra}). Following the proof of  Theorem \ref{thm:gen weight} verbatim, it also follows that if $Q_{i,j}(x)$ are log-concave sequences and $w_i(x)$ are ultra-log-concave sequences, then for all $i\in[n]$, the probability mass functions of $h_i$ are ultra-log-concave sequences. In fact, for any positive sequence $f(k)$, if the weight function $w(k)$ is such that $w(k)f(k)$ is log-concave, then it can also be shown easily that $\PP_{n,w,Q}(h_i=k)f(k)$ is log-concave in $k$, for all $i\in [n]$. 
\end{remark}

The following are a few examples of the discrete orthogonal polynomial ensembles ($Q_{i,j}(x)=Q(x)=x^2$ and $w_{i,j}(x)=w(x)$) that are well-studied \cite{KJ02}.

\smallskip
\textbf{Meixner ensemble:} For $m\geq n$ and $q\in[0,1]$ with $x\in\NN$, the weights $w(x)={ \binom{x+m-n}{x}}q^{x}$ in \eqref{eq: DOPE} gives us the measure $\mathbb{P}_{n,m,\mbox{Me}}$ on $\overrightarrow{\NN}^n$, known as Meixner ensemble.

\smallskip
\textbf{Charlier ensemble:} For $\alpha>0$ and $x\in\NN$, the weights $w(x)=e^{-\alpha}\frac{\alpha^{x}}{x!}$ gives us the measure $\mathbb{P}_{n,\alpha,\mbox{Ch}}$ on $\overrightarrow{\NN}^n$, known as Charlier ensemble.  

\smallskip
\textbf{Krawtchouk ensemble:} For $p\in(0,1)$ and $q=1-p$ with $K\in\NN$ and $K\geq n$, the weights $w(x)={\binom{K}{x}}p^{x}q^{K-x}$ where $x\in \mathbb{K}:=\{0,1,\dots,K\}$, gives us the measure $\PP_{n,K,p,\mbox{Kr}}$ on $\overrightarrow{\mathbb{K}}^n$, known as Krawtchouk ensemble. 

\smallskip
\textbf{Hahn ensemble:} For integers $a,K$ with $K\geq a\geq n$ and $K=a+n-1$, the weights $w(x)={\binom{x+a-n}{x}}\binom{K+a-n-x}{K-x} $ where $x\in\mathbb{K}$, gives us the measure
on $\overrightarrow{\mathbb{K}}^n$ known as Hahn ensemble.

\smallskip

In our next example, $Q(x)$ behaves like $x^{2\theta}$ for large $x$, and provides discrete analogues of $\beta$-log gases.

\textbf{Integrable discrete beta ensembles:} We now consider the probability measure, $\PP_{n}^{\theta,m}$ on $\overrightarrow{\Z}^{n,m,\theta}$ where,      
\begin{align}
  &  \overrightarrow{\Z}^{n,m,\theta}=\{(\lambda_1\geq\lambda_2\geq\dots\geq\lambda_n):\lambda_i\in \N \mbox{ and } \lambda_1\leq m+(n-1)\theta\}\nonumber,\\ \label{eq:anotherformofdiscretebetagas}
&    \PP_{n}^{\theta,m}(\lambda_1\geq\lambda_2\geq\dots\geq\lambda_n):=\frac{1}{Z_{n,m,\theta}}\prod\limits_{1\leq i<j\leq n}Q_{\theta}(\lambda_i-\lambda_j+(j-i)\theta)\prod\limits_{j=1}^{n}w(\lambda_j+(n-j)\theta),\\
 &   Q_{\theta}(x):=\frac{\Gamma(x+1)\Gamma(x+\theta)}{\Gamma(x+1-\theta)\Gamma(x)}.\nonumber
\end{align}

Here $\theta>0$ and $m\in[0,\infty]$. The weight function $w(x)$ is assumed to be positive and continuous for $x\in[0,m+(n-1)\theta]$. For $m=\infty$ case, $w(x)$ has to be decaying fast enough for $Z_{n,m,\theta}<\infty$. 
Such measures were introduced in ~\cite{BGG17} and extensively studied, due to their connections to discrete Selberg integrals and integrable probability (see Section $1$ of \cite{BGG17}). Note that for $\theta=1$ and $\theta=1/2$, we get \eqref{eq: DOPE} for $Q(x)=x^2$ and $Q(x)=x$ respectively. Note that above measure can be seen as a special case of \eqref{eq: DOPE}. Following the proof idea of Theorem \ref{thm:gen weight}, we can also show that the distribution of $\lambda_1$ under the measure $\PP_{n}^{\theta,m}$, with $w(x)$ log-concave, is log-concave. It was shown in \cite{GH19} that, if $\theta=\beta/2$ and for all $\beta\geq 1$, after appropriate scaling and centering $\lambda_1$ converges to $TW_{\beta}$. As log-concavity is preserved under scaling, centering and weak limit (Lemma \ref{lem: Discrete to continuous}), it follows that $TW_{\beta}$ is log-concave (for $\beta\geq 1$). We shall show later that log-concavity of $TW_{\beta}$ holds for all $\beta>0$ (Corollary \ref{cor: Stoch Airy operator}).

Although the above ensembles are usually defined without the ordering on $h_i$s, we order $h_i$s as we are interested in studying the rightmost elements. 
In all four examples mentioned above, $w(x)$ is easily seen to be log-concave. Hence we get the following result immediately from Theorem \ref{thm:gen weight}.
\begin{corollary}\label{cor: MeixnerCharlierrtc}
 Meixner, Charlier, Krawtchouk and   Hahn ensembles are discrete  log-concave distributions. Their one-dimensional marginals are log-concave on $\NN$. In particular, this is true for the largest points in these ensembles.
\end{corollary}

Note that in the above examples, the weights are ultra-log-concave for Charlier and Krawtchouk ensembles. Following Remark \ref{rem: ulc}, the distribution of $h_i$ is ultra-log-concave for these cases. By Theorem $1.1$ and  \cite[Proposition $1.2$]{AMM22} the following corollary which gives Poisson concentration bounds is immediate. Let $a(x):=2\frac{(1+a)\log(1+a)-a}{a^2}$ for $a\in[-1,\infty)$. 

\begin{corollary}\label{cor: Poisson concentration} 
    Let $h_i$ be the one-dimensional marginals of Charlier and Krawtchouk ensembles. Then these random variables satisfy the following bounds.
    \begin{itemize}
        \item $\mathbb{P}\l(h_i-\E [h_i]\geq t\r)\leq \exp\l(-\frac{t^2 }{2\E [h_i]}a\l(\frac{t}{\E [h_i]}\r)\r)$ for all $t\geq 0$.
        \item $\mathbb{P}\l(h_i-\E [h_i]\leq -t\r)\leq \exp\l(-\frac{t^2 }{2\E [h_i]}a\l(-\frac{t}{\E [h_i]}\r)\r)$ for  $0\leq t \leq \E[h_i]$.
        \item $Var(h_i)\leq \E[h_i]$.
    \end{itemize}
\end{corollary}

\subsection{Log-concavity in Schur measures} Schur measures are another well-studied class of ensembles on $\Z$ that contain the Meixner and other ensembles, although they correspond only to $\beta=2$ case. They are defined using Schur polynomials $s_{\lambda}(x)$ defined for $\lambda\in \Lambda=\bigcup_n \Lambda_n$ and variables $x=(x_1,x_2\ldots)$ by
\begin{align}
    s_{\lambda}(x)=\sum_{T}x^T
\end{align}
where the sum is over semi-standard Young tableau $T$ of shape $\lambda$ and $x^T=\prod_{i}x_i^{t_{i}}$ where $t_i$ is the number of times $i$ occurs in $T$ (see  \cite[Section I.$3$]{macdonald} for details on Schur polynomials).

Given parameters $a=(a_1,a_2,\ldots)$ and $b=(b_1,b_2,\ldots )$ with $a_i,b_i\in \C$, the corresponding Schur measure on $\Lambda$ is defined by (see \cite{okounkov} or  \cite[Section~3]{johanssonnonintersecting})
\begin{align*}
   \mathbb P_{a,b}(\lambda)=\frac{1}{Z_{a,b}}s_{\lambda}(a)s_{\lambda}(b).
\end{align*}
In general, $\mathbb P_{a,b}(\lambda)$ is a complex measure. It is a probability measure under either of the following conditions:
\begin{enumerate}
    \item $a_i\ge 0$ and $b_i\ge 0$ for all $i$.
    \item $b_i=\overline{a}_{\sigma(i)}$ for all $i$, for some bijection $\sigma$ of $\{1,2,\ldots \}$ to itself.
\end{enumerate}
We shall be concerned with the first case.  

One may regard $\lambda\in \Lambda$ as a partition or as a collection of weakly ordered particles $\lambda_1\ge \lambda_2\ge \ldots $ We show that the distribution of each $\lambda_i$ is log-concave.
\begin{theorem}\label{thm:schurmeasures}
Assume that $a_i\ge 0$ and $b_i\ge 0$ for all $i$. All one dimensional marginals of the Schur measure $\mathbb P_{a,b}$ are discrete log-concave. 
\end{theorem}

For the choice $a=b=(\sqrt{\alpha},\sqrt{\alpha},\sqrt{\alpha},\dots,\sqrt{\alpha},0,0,\dots)$ with zeros after $n$ many entries, we have
\begin{align*}
    \mathbb P_{a,b}(\lambda)=(1-\alpha)^{n^2}\alpha^{|\lambda|}|\mbox{semi-standard Young tableaux of shape }\lambda \mbox{ with entries in } [n]|^2.
\end{align*}
This is a mixture of $z$- measures (which are Plancherel-like measures that arise in the representation theory of certain non-commutative groups) on partitions of a fixed number $n=|\lambda|$ by the negative binomial distribution on $n=0,1,2,\dots$ with parameter $\alpha$; see \cite[Section 2.1.4]{okounkov}, \cite{borodin2005z} and \cite{BO05} for details. One can also obtain Poissonized Plancherel measure on the set of partitions as a special case of Schur measures (see Section $2.1.4$ of \cite{okounkov}).

An important probability context in which Schur measures arises is that of last passage percolation. Let $w_{i,j}$ be independent random variables with Geometric distribution $\mathbb{P}\{w_{i,j}=k\}=(1-a_ib_j)(a_ib_j)^k$, $k\ge 0$. Define the passage time from $\mathbf{1}=(1,1)$ to $\mathbf{n}=(n,n)$  by   
\begin{align*}
L_n^{\square}:=\max\limits_{\gamma}\ell(\gamma) \qquad \mbox{ where }\ell(\gamma)=\sum\limits_{v\in \gamma}\zeta_v,
\end{align*}
and the maximum is  over all up/right oriented paths $\gamma$ in $\mathbb{Z}^2$ from $\mathbf{1}$ to $\mathbf{n}$. It is a well-known result that under $\mathbb P_{a,b}$, the rightmost particle $\lambda_1$  has the same distribution as $L_n^{\square}$ (see \cite{johanssonnonintersecting}). Then, Theorem~\ref{thm:schurmeasures} implies that $L_n^{\square}$ has log-concave distribution.

Certain choices of $a_i,b_i$ and additional symmetry constraints are of particular interest. We mention three of these, see \cite{FR07} for details.
\begin{enumerate}
\item Let $w_{i,j}$ be i.i.d. with $\mbox{Geo}(1-q)$ distribution (so $a_i=b_i=\sqrt{q}$). Then the last passage time $L_n^{\square}$ is denoted $G^{(2)}_{\mathbf{1},\mathbf{n}}$. 
    \item Let $w_{i,j}=w_{j,i}$ be otherwise independent, and have $\mbox{Geo}(1-q)$ distribution when $i\not=j$ and $\mbox{Geo}(1-\sqrt{q})$ distribution when $i=j$. The passage time from $(1,1)$ to $(n,n)$ is denoted $G^{(4)}_{\mathbf{1},\mathbf{n}}$.
    \item Fix $n$ and let $w_{i,j}=w_{n+1-i,n+1-j}$ be otherwise independent and have $\mbox{Geo}(1-q)$ distribution when $i+j\le n$ and $\mbox{Geo}(1-\sqrt{q})$ distribution when $i+j=n+1$. The passage time from $(1,1)$ to $(n,n)$ is denoted $G^{(1)}_{\mathbf{1},\mathbf{n}}$.
\end{enumerate}

Although Theorem~\ref{thm:schurmeasures} does not directly apply to the second and third situations, the proof of Theorem~\ref{thm:schurmeasures} carries over easily to cover these cases.
\begin{corollary}\label{thm:LPP models}
    $G_{\mathbf{1},\mathbf{n}}^{(1)}, G^{(2)}_{\mathbf{1},\mathbf{n}}, G^{(4)}_{\mathbf{1},\mathbf{n}}$ are log-concave distributions.
\end{corollary}
\begin{remark}
    One can also view this as a corollary of Theorem~\ref{thm:gen weight}. Indeed, the distribution of $G^{(\beta)}_{\mathbf{1},\mathbf{n}}$ for $\beta=2,1$ and $4$ is exactly the same as that of $h_n$ in \eqref{eq: DOPE} with $Q(x)=x^{2}, x$ and $x$ respectively with $w(x)=q^x, q^{x/2}$ and $q^{x/2}$ respectively (see Proposition $1.3$ of \cite{KJ00}, Lemma $3.2$ of \cite{Baik02} and Equations $4.6$ and $5.6$ of \cite{FR07}).  If $G_{1,m,n}^{(2)}$ denotes last passage time from $(1,1)$ to $(m,n)\in\Z^2$, it can also be shown that $G_{1,m,n}^{(2)}$ is log-concave. Using the Geometric limit to exponentials, log-concavity of passage times for exponential weights also follows.
\end{remark}
We can now describe the proof methods of the  main results stated till now.  In the next section we shall see log-concavity results for ordered elements in continuum ensembles. Those proofs are almost trivial, due to the availability of the Pr\'ekopa-Leindler inequality (which asserts that log-concavity is preserved under marginalization) as explained there. 
A key contribution of our work is to identify the right notion of discrete log-concavity, by recognizing that a recent inequality of Halikias-Klartag-Slomka \cite{HKS21} works as a discrete analogue of the Pr\'ekopa-Leindler inequality in a way that is perfectly tuned to be of use for our purposes.
Indeed, the proof of Theorem~\ref{thm:gen weight} becomes analogous to the proofs in the continuum setting, with the Halikias-Klartag-Slomka inequality replacing the 
Pr\'ekopa-Leindler inequality. 

A well known result, due to Johansson \cite{KJ00}, states that the limiting distribution of largest particle in Meixner ensemble with $q=\alpha/n^2$, converges to the length of the top row under Poissonized Plancherel measure. Theorem \ref{thm: Poisson Plancherel log-concave} is proved by generalizing this fact (that corresponds to $\beta=2$) to all $\beta>0$. This generalization is stated as Theorem~\ref{thm: Meixner limit to Plancherel} later.

\subsection{Remarks on the definition of discrete multivariate log-concavity}\label{sec: disc mult lc}

In the discrete setting, there are multiple inequivalent definitions of multivariate log-concavity in the literature (see \cite{Murota-Shioura}). Our definition~\ref{def: disc lc} is designed so as to interpret a recent inequality of Halikias, Klartag and Slomka~\cite{HKS21} (see also \cite{KL18} and \cite{GRST21} for some variants) as preservation of log-concavity under marginalization. In fact, Halikias-Klartag-Slomka prove a whole family of inequalities, of which we have chosen the most natural-looking one (the choice is $\lambda=\frac12$ in \cite[Theorem~1.2]{HKS21}). In the language of Murota \cite{Mur03:book}, a natural alternative nomenclature for our chosen notion of log-concavity on $\Z^n$ would be $L^{\#}$-log-concavity, although we do not use this terminology for simplicity.

We first explain the result of Halikias-Klartag-Slomka \cite{HKS21}, which is the main tool in the proofs.
\begin{result}[Theorem $1.2$ of \cite{HKS21}]\label{thm: Klartag result}
Let $s\in[0,1]$ and suppose that $f,g,h,k:\Z^{n}\rightarrow[0,\infty)$ satisfy
\begin{align}\label{eq:Klartag conditions}
f(x)g(y)\leq h\left(\left\lfloor s x+(1-s)y\right\rfloor\right)k\left(\left\lceil(1-s) x+s y\right\rceil\right)\quad \forall x,y\in \Z^n
\end{align}
where $\left\lfloor x\right\rfloor=\left(\left\lfloor x_1\right\rfloor, \left\lfloor x_2\right\rfloor,\dots, \left\lfloor x_n\right\rfloor\right)$ and $\left\lceil x\right\rceil=\left(\left\lceil x_1\right\rceil, \left\lceil x_2\right\rceil,\dots, \left\lceil x_n\right\rceil\right)$. Then
\begin{align*}
    \left(\sum\limits_{x\in \Z^n}f(x)\right) \left(\sum\limits_{x\in \Z^n}g(x)\right)\leq \left(\sum\limits_{x\in \Z^n}h(x)\right)\left(\sum\limits_{x\in \Z^n}k(x)\right).
\end{align*}
\end{result}\
To relate this to the Pr\'ekopa-Leindler inequality, fix some $k<n$ and define 
\begin{align*}
f_{(x_{k+1},\ldots ,x_n)}(t_1,\ldots ,t_k)&:=f(t_1,\ldots ,t_k,x_{k+1},\ldots ,x_n), \\ \hat{f}(x_{k+1},\ldots ,x_n)&:=\sum_{t_1,\ldots ,t_k} f_{(x_{k+1},\ldots ,x_n)}(t_1,\ldots ,t_k).
\end{align*}
If $f,g,h,k$ satisfy \eqref{eq:Klartag conditions}, then so do $f_{(x_{k+1},\ldots ,x_n)},g_{(y_{k+1},\ldots ,y_n)},h_{(u_{k+1},\ldots ,u_n)},k_{(v_{k+1},\ldots ,v_n)}$, provided $u_j=\left\lfloor sx_j+(1-s)y_j\right\rfloor$ and $v_j=\left\lceil sx_j+(1-s)y_j\right\rceil$. The conclusion of Result~\ref{thm: Klartag result} is precisely that $\hat{f},\hat{g},\hat{h},\hat{k}$ also satisfy \eqref{eq:Klartag conditions}.  

This was the motivation for  Definition~\ref{def: disc lc}, which  says that $f:\Z^n\to [0,\infty)$ is log-concave if the condition~\eqref{eq:Klartag conditions} holds for $s=\frac12$ and $f=g=h=k$. By the above discussion, we get
\begin{corollary}\label{cor:HKSmarginals} If $f$ is a log-concave probability mass function on $\Z^n$, then all its lower dimensional co-ordinate  marginals are also log-concave.
\end{corollary}
However, unlike in the continuum setting, marginals in other directions need not be log-concave. For example, if $f:\Z^2\to \R$ is defined by 
\begin{align*}
    f(0,0)=f(1,1)=3 \mbox{ and }f(1,0)=f(0,1)=1
\end{align*}
and $f(i,j)=0$ for all other $(i,j)$. Then $f$ is log-concave. The marginal in the $(1,1)$ direction defined by  $g(t)=\sum_xf(x,t-x)$ for $t\in \Z$ satisfies  $g(0)=g(2)=3$, $g(1)=2$ and $g(i)=0$ for all other $i$. Clearly $g$ is not log-concave. 

In fact, log-concavity as per our definition is not preserved under $\frac{\pi}{2}$-rotation  of the lattice $\Z^2$. The same example $f$ as above shows this.

\section{Results in the continuum setting}
\subsection{Log-concavity in continuum Coulomb gas ensembles}

Several interacting particle systems in statistical mechanics such as Coulomb gases, Ising model, exclusion processes, are modelled by Gibbs measures \cite{FV17}. Consider the Gibbs measure determined by positive temperature parameter $\beta\in(0,\infty)$ and a Hamiltonian function $\mathcal{H}_n:\R^n\to \R\cup \{\infty\}$ of $n$ real-valued variables 
$x=(x_1,x_2,\dots,x_n)$, given by
\begin{align}\label{eq: unordGibbs measures}
    d{\PP}_{\HH_n,\beta}(x)&\propto 
    {\exp\left\{-\beta\mathcal{H}_n(x_1,\dots,x_n)\right\}}dx_1\ldots dx_n.
\end{align}
One-dimensional $\beta$-Coulomb gases are special cases of \eqref{eq: unordGibbs measures} given by 
\begin{align}\label{eq:betacoulombgas}
    \mathcal{H}_n(x_1,\dots,x_n)=-\sum\limits_{i<j}\log|x_i-x_j|+\sum\limits_{i}V(x_i),
\end{align}
 where $V:\R\to \R\cup \{\infty\}$ is function that increases fast enough at $\pm \infty$ to ensure integrability of $d\PP_{\HH_n,\beta}(x)$. When $V$ is quadratic and $\beta=1,2,4$, the $\beta$-Coulomb gas is the joint law of eigenvalues in Gaussian orthogonal, unitary and symplectic ensembles respectively (see \cite{AGZ} for more about Gaussian ensembles). 

Although  the usual definitions of $\beta$-ensembles have $x_i$ unordered, our interest is in the ordered variables. The largest variable is often of particular interest (e.g., in the case of the Gaussian ensembles mentioned above, this would be the largest eigenvalue of a random matrix drawn from the ensemble). If the Hamiltonian $\mathcal{H}_n:\R^n\to \R\cup \{\infty\}$ of the system \eqref{eq: unordGibbs measures} is symmetric (with respect to arbitrary permutations of the coordinates), observe that the behavior of the order statistics of the random vector $X$ drawn from ${\PP}_{\HH_n,\beta}$ coincides with the behavior of the system 
\begin{align}\label{eq: Gibbs measures}
    d\overrightarrow{\PP}_{\HH_n,\beta}(x)&=\frac{1}{Z_{\HH_n,\beta}}{\exp\left\{-\beta\mathcal{H}_n(x_1,\dots,x_n)\right\}} \ind_{\mathcal{W}_n}(x) dx_1\ldots dx_n,
\end{align}
where  
$\mathcal{W}_n=\{y\in \R^n\suchthat y_1<\ldots <y_n\}$ is the Weyl chamber.

We are now in a position to formulate our key observation about log-concavity in the continuous setting. 

\begin{theorem}\label{thm: continuous log-concave} 
Consider the system \eqref{eq: Gibbs measures}, with the Hamiltonian $\mathcal{H}_n$ of form \eqref{eq:betacoulombgas}.
Suppose that $V:\R\to \R\cup \{\infty\}$ is convex. Then:
\begin{enumerate}[(1)]
\item The $\beta$-Coulomb gas  $\overrightarrow{\PP}_{\HH_n,\beta}$ is a log-concave distribution on $\R^n$.

\item The ordered points $x_{k}$ and the gaps $x_k-x_{k-1}$ of the $\beta$-Coulomb gas  have log-concave distributions on $\R$. 
\end{enumerate}
\end{theorem}

The first statement is not new-- it was already observed in the Ph.D. thesis of Wang \cite{Wang14}, and also by Chafai and Lehec~\cite[Lemma~2.5]{Chafai-Lehec}.

As sums of convex functions composed with linear maps are convex, we obtain the first part of Theorem \ref{thm: continuous log-concave}. Using the Pr\'ekopa-Leindler inequality \cite{Leindler72, Prekopa71,Prekopa73}, which implies that the marginals of log-concave distribution are log-concave, the second part of Theorem \ref{thm: continuous log-concave} follows.

A somewhat related notion is that of log-supermodularity (also called $\mbox{MTP}_2$). A probability density $f$ on $\mathbb{R}^n$ is said to be log-supermodular (i.e., $\log f$ is supermodular as defined in \cite[Definition 2.3]{FMZ24}) if 
\begin{align*}
    f(x)f(y)\leq f(x\wedge y)f(x\vee y), \mbox{ for all } x,y\in \R^n,
\end{align*}
where $x\wedge y$ and $x\vee y$ are the componentwise minimum and maximum respectively.  One implication of log-supermodularity is positive association (thanks to the FKG inequality, see \cite{FKG71, Preston74}),  which is difficult to prove otherwise. 

\begin{theorem}\label{thm: supermodularity}
Consider the system \eqref{eq: Gibbs measures}, with the Hamiltonian $\mathcal{H}_n$ of form \eqref{eq:betacoulombgas}.
For any $V$ in \eqref{eq:betacoulombgas} and any $\beta>0$, the density of $\overrightarrow{\PP}_{\HH_n,\beta}$ is log-supermodular. In particular, the points of the $\beta$-Coulomb gas are positively associated.
\end{theorem}

The proof is a direct computation using only the elementary inequality,
\begin{align*}
    (x_2-x_1)(y_2-y_1)\leq (x_2\vee y_2\ -\  x_1\vee y_1)(x_2\wedge y_2\ -\  x_1\wedge y_1)
\end{align*}
for any $x_1<x_2$ and $y_1<y_2$. Alternately one can check the derivative condition in \cite[Proposition 2.5]{FMZ24}. 

It is well-known that when $V(x)=x^2$, the distribution of $x_{n}$, after appropriate shifting and scaling, converges to $\mbox{TW}_{\beta}$, the $\beta$ version of Tracy-Widom distribution. For special values of $\beta=1,2,4$ this was proved by Tracy and Widom \cite{TW94}, and the case of general $\beta$ was proved by Ramirez-Rider-Vir\'{a}g~\cite{RRV06}, who defined $\mbox{TW}_{\beta}$ as the distribution of the smallest eigenvalue of the stochastic Airy operator
\begin{align*}
    \mathcal{H}_{\beta}=-\frac{d^2}{dx^2}+x+\frac{2}{\sqrt{\beta}}b'_x \;\;\; (\mbox{here }b\mbox{ is standard Brownian motion}) 
    \end{align*}
    acting on an appropriate Hilbert space (see \cite{RRV06} for details). Note that log-concavity and log-supermodularity are preserved under shifting, scaling and under weak limits ( at least if non-degenerate). As non-degenerate log-concave measures have density, we immediately get the following corollary.

\begin{corollary}\label{cor: Stoch Airy operator}
 Fix $\beta>0$.
\begin{enumerate}
    \item $TW_{\beta}$ distribution has a density and the density function is log-concave.
    \item For any $k\geq 1$, the smallest $k$ eigenvalues of $\mathcal{H}_{\beta}$ have log-concave and log-supermodular joint density and hence are positively associated.
\end{enumerate}
\end{corollary}

Observe that much more is true: As the joint distribution of largest $k$ eigenvalues of $\beta$-ensemble with quadratic potential is log-concave, the same is true of the $k$ smallest eigenvalues of $\mathcal{H}_{\beta}$. Therefore, the gaps among the smallest $k$ eigenvalues of $\mathcal{H}_{\beta}$ are also jointly log-concave. Further, in the Laguerre/Wishart ensembles (take $V(x)=\frac{x}{2}+\l(\frac{1}{\beta}-\frac{a+1}{2}\r)\log x$ for $x>0$ in \eqref{eq:betacoulombgas}, where the parameter $a>-1$),  the smallest $k$ eigenvalues have a joint log-concave distribution, by Theorem \ref{thm: continuous log-concave}. Again taking weak limits,  we deduce that the joint distribution of $(\Lambda_{0}(\beta,a),\dots,\Lambda_{k-1}(\beta,a))$, the $k$ smallest eigenvalues of the {\em stochastic Bessel operator} (as defined in \cite{RR09}) is log-concave for $a>\frac{2}{\beta}-1$. 

Although $TW_{\beta}$ distributions are widely studied, the log-concavity property does not seem to have been noticed before. Here are some consequences that follow immediately from log-concavity, but could be difficult to prove otherwise. 
\begin{enumerate}
    \item That $\mbox{TW}_{\beta}$ has a density appears to have not been shown before (for $\beta\not\in \{1,2,4\}$). But any non-degenerate log-concave measure has density by Borell's characterization, hence Corollary~\ref{cor: Stoch Airy operator} implies that $\mbox{TW}_{\beta}$ has a density. The same applies to joint distributions of the smallest $k$ eigenvalues of $\mathcal{H}_{\beta}$ and those of the stochastic Bessel operator mentioned above.
    \item Tail bounds on $\mbox{TW}_{\beta}$ (see  \cite[Theorem~1.3]{RRV06}) trivially transfer to corresponding pointwise bounds on the density of $\mbox{TW}_{\beta}$.
    \item Further, the convergence results can be strengthened. For any $k\in \N$, the joint density $f_{\beta}$ of the smallest $k$ eigenvalues of $\mathcal{H}_{\beta}$ is log-concave. Let $f_{n,\beta}$ be the joint density of the vector $\l(n^{1/6}\l(2\sqrt{n}-\lambda_{\beta,\ell}\r)\r)_{\ell\in[k]}$ as in  \cite[Theorem $1.1$]{RRV06}. By  \cite[Proposition $2$]{CS10}, we have the following corollary strengthening the result of Ramirez-Rider-Vir\'{a}g \cite{RRV06}.
    \begin{corollary}\label{cor: uniform conv of TW densities}
    For any $\beta>0$, there exists some $a_0>0$ such that for all $a<a_0$, we have
    \begin{align*}
        \sup\limits_{x\in\R^k}e^{a\lVert x\rVert}\left|f_{n,\beta}(x)-f_{\beta}(x)\right|\rightarrow 0.
    \end{align*}
    \end{corollary}
    \item By Theorem \ref{thm: continuous log-concave} we have that the distributions of largest eigenvalues of Hermite and Laguerre $\beta$-ensembles (see \cite{LR10} for details), are log-concave for all $\beta>0$. The fluctuations of these eigenvalues are known to converge weakly to $TW_{\beta}$ (see Equation $1.3$ and $1.5$ of \cite{LR10}). By  \cite[Corollary $6$]{meckes2014equivalence}, Corollary \ref{cor: Stoch Airy operator} yields the following corollary.
\begin{corollary}\label{cor: conv of Her-Lag moments}
    For all $\beta>0$ and for all $k\in\mathbb{N}$, the $k$-th moments of the largest eigenvalues of Hermite and Laguerre ensembles converge weakly to the corresponding moments of $TW_{\beta}$. 
\end{corollary}
The above result was known only for $\beta\geq 1$ (see Corollary $3$ of \cite{LR10}). Log-concavity could also have other applications. For example, the partial result of\ \ \   log-concavity of \cite{BLS17} was used in \cite{BCD23}.
    \item Tracy and Widom \cite{TW94} had also computed expressions for  ``higher-order Tracy-Widom laws'', which emerge as limiting distributions for the $k$-th largest eigenvalue of the GUE. These also exhibit universality; for example, Baik, Deift and Johansson \cite{BDJ00} showed that the length of the second row of a Young diagram under the Plancherel measure also converges (after centering and scaling) to the same second-order Tracy-Widom law. While the expressions for the higher-order laws are even less tractable, their log-concavity is an immediate consequence of our results. Moreover, the log-concavity and log-supermodularity of the smallest $k$ eigenvalues of stochastic Airy operator,
    which would possess Tracy-Widom laws of various orders as marginals, is also an automatic consequence.
\end{enumerate}

\begin{remark}
      In \cite{BLS17} a much more involved proof (the authors attribute the proof to P. Deift) is presented to show that $\mbox{TW}_2$ is log-concave on the positive half of the real line. That proof uses a different description of the $\mbox{TW}_2$ distribution in terms of the solutions to the  Painl\'{e}ve-II differential equation (this was in fact the original description given by Tracy and Widom). Although more involved, the technique is very different and has potential future uses. For example, {the method could be useful} in studying higher order analogues of $TW_2$ described in terms of solutions of higher order equations of the Painl\'{e}ve-II hierarchy (See \cite{DMS18}). Hence, for the sake of completeness, in  Appendix \ref{appendix} we present a modification of Deift's proof and  show the log-concavity of $\mbox{TW}_2$ on the whole of the real line.
\end{remark}

\begin{remark}
    A probability density $f$ on $\mathbb{R}^n$ is said to be strongly log-concave with parameter $\sigma^2$, if $f(x)/\phi_{\mu,\sigma^2}$ is log-concave function, where $\phi_{\mu,\sigma^2}$ is probability density of $N(\mu,\sigma^2 I_n)$ random vector. The arguments in the proof of Theorem \ref{thm: continuous log-concave} also give that the ordered points $x_{k}$ of $\beta$-Coulomb gases with $V(x)=x^2$ are strongly log-concave with parameter $(0,1/2\widehat{\beta})$ for any $\widehat{\beta}<\beta$. As strong log-concavity is preserved under the limit (with common parameters), one might hope for strong log-concavity of $\mbox{TW}_{\beta}$. But after appropriate scaling and shifting of $x_{n}$, the resulting random variables which converge to $TW_{\beta}$ are strongly log-concave with parameter $(-2, n^{1/3}/2\widehat{\beta})$. As there is no common parameter, the strong log-concavity in the limit is not guaranteed. In fact, $\PP(TW_{\beta}>x)\sim \exp(-2\beta x^{3/2}/3)$ as $x\rightarrow\infty$ (by \cite{RRV06}). Hence $TW_{\beta}$ cannot be strongly log-concave.
\end{remark}

    Another useful feature of log-concave distributions in the context of information theory is that one obtains bounds on a few important characteristics of distributions such as Shannon and R\'enyi entropies \cite{RE08}. For a random variable $X$ with density function $f$, the R\'enyi entropy of order $\alpha\in(0,\infty) \setminus\{1\}$, is defined as
\begin{align*}
    h_{\alpha}(X)=\frac{1}{1-\alpha}\log\left(\int f^{\alpha}(x)dx\right),
\end{align*}    
assuming the integral exists. For $\alpha\rightarrow 1$ one obtains the usual Shannon differential entropy $h(X)=-\int f\log f.$ It is well known that the entropy among all zero-mean random variables with the same second moment is maximized by the Gaussian distribution:
\begin{align*}
    h(X)\leq \log\left(\sqrt{2\pi e \Var(X)}\right).
\end{align*}

Although one cannot hope for a lower bound for entropy in general, it was shown in \cite{BM11} that in the class of log-concave random variables, the above inequality can be reversed. A recent result in \cite{MNR23} shows that, for any log-concave random variable $X$, we have the sharp inequality
\begin{align*}
    h(X)\geq \frac{1}{2}\log\left(\Var(X)\right)+1.
\end{align*}

The work of \cite[Theorem IV.$1$]{BM11} (cf. \cite{FMW16, FLM20}) and \cite[Corollary $1.2$]{MNR23} gives sharp lower bounds on the R\'enyi entropies for log-concave random variables in terms of maximum density and variance respectively. Using the fact that $TW_{\beta}$ are log-concave, these results can be used to obtain bounds on R\'enyi entropies and Shannon entropy of $TW_{\beta}$ distributions, provided one obtains bounds on the variance of these distributions. With variance bounds and log-concavity of $TW_{\beta}$ distributions, one can also obtain bounds on higher central moments, using the work of \cite[Proposition $1$]{MK17}. Although we are not aware of theoretical bounds on the moments of $TW_{\beta}$ distributions, there exist algorithms to compute the moments numerically \cite{TWN21}. 

\subsection{Log-concavity of \texorpdfstring{$\mathcal{A}_2$}  \ \ process}
We study log-concavity of $\mathcal{A}_2$  process ($\mbox{Airy}_2$ process) which is one of a central object in random matrix theory and last passage percolation. The $\mathcal{A}_2$  process was introduced by Pr\"{a}hofer and Spohn \cite{PS02} in the study of the scaling limit of a discrete polynuclear growth model.

Consider a collection of $N$ Brownian bridges $\left( B_1(t),\dots, B_N(t)\right)$, all starting from zero at time $t=0$ and ending at zero at time $t=1$, and conditioning them not to intersect in the region $t\in(0,1)$. We will always assume that the paths are ordered so that $B_1(t)<\hdots<B_N(t)$ for $t\in(0,1)$.
The relation between the $\mbox{Airy}_2$ process and non-intersecting Brownian bridges lies in the fact that, suitably rescaled, the top path of a collection of non-intersecting Brownian bridges converges to the $\mbox{Airy}_2$ process minus a parabola:
\begin{align}\label{eq: BM to Airy process}
    2N^{1/6}\left(B_N\left(\frac{1}{2}(1+n^{-1/3}t)\right)-\sqrt{n}\right)\rightarrow \mathcal{A}_2(t)-t^2
\end{align}
in the sense of convergence in distribution in the topology of uniform convergence on compact sets (See Equation $1.6$ of \cite{nR15}). This result is
well-known in the sense of convergence of finite-dimensional distributions; the stronger convergence stated here was
proved in \cite{CH14}. We prove the following theorem.

\begin{theorem}\label{thm: Airy_2 process}
    For any $k\geq 1$ and $t_1<\dots<t_k$, the joint distribution $\left(\mathcal{A}_2(t_1),\dots, \mathcal{A}_2(t_k)\right)$ is log-concave.
\end{theorem}

The definition of log-concave measures also extends to locally convex Hausdorff spaces \cite{Bor74,BM16}. Consider the distribution of the $\mathcal{A}_2$ process on the function space $C(\R)$ equipped with the  topology of uniform convergence on compact sets. Using \cite[Theorem $2.1$]{Bor74}, Theorem \ref{thm: Airy_2 process} implies that the (infinite-dimensional) distribution of the $\mathcal{A}_2$ process is a log-concave measure.

\begin{remark}
    It is known that the long time limit of $n$ spatial points in the solution of KPZ equation for the sharp wedge initial conditions are exactly the finite dimensional distributions of $\mbox{Airy}_2$ process \cite{PS11}. As a result we have that the finite dimensional distributions of KPZ solutions converge to a log-concave distribution. One could also study whether for a fixed time, the joint distribution of $n $ spatial point in KPZ solutions are log-concave. 
\end{remark}

 If one prefers the stationary process $\mathcal A_2(t)-t^2$, observe that its distribution is just a translation of the distribution of $\mathcal A_2$ on $C[0,1]$, hence it is also log-concave. As $\mathcal{A}_2(t)$ is distributed as $TW_2$ for any fixed $t$, this provides another proof for log-concavity of $TW_2$. By adaptation of the proof of Theorem \ref{thm: Airy_2 process}, it follows that Theorem \ref{thm: Airy_2 process} can be extended to finite dimensional distributions (and hence the process distribution) of any line from the Airy line ensemble \cite{CH14}. Since the initial version of our work appeared, Wu \cite{XW24} obtained log-concavity of finite dimensional distributions (and hence the process distribution) of more general line ensembles.
 
 As $\mathcal{A}_2(t)$ is an important object in modern probability, the observation of log-concavity of its finite distributions may have several implications. We remark one such result here. Let $B $ be a convex, open symmetric set in the state space of Airy-2 process and let $C$ be the scaling $C=(\frac{2}{a}-1)B$ where $0<a<1$, then 
 $$
 \PP(\mathcal{A}_2(\cdot) \notin B)\geq \PP(\mathcal{A}_2(\cdot) \notin C)^a.
 $$
 This follows from Theorem $3$ of Bobkov and Melbourne~\cite{BM15}.
 
 The proof of Theorem \ref{thm: Airy_2 process} involves restricting Gaussian density (which is log-concave) to an appropriate convex set, which preserves log-concavity. This idea is of wider applicability. To illustrate, we now prove the log-concavity of the Airy distribution.

Let $(B^{\mb{\tiny ex}}(t))_{t\in[0,1]}$ be the Brownian excursion. The Airy distribution is the distribution of the area under the Brownian excursion, i.e., of the random variable $A:= \int_{0}^{1}B^{\mb{\tiny ex}}(t)dt$. In the context of random interfaces, it is the distribution of maximal height of fluctuating interface in $(1+1)$ dimensional Edwards-Wilkinson model \cite{MC09}. It also shows up in combinatorics, in particular the limiting distribution of fluctuations/area of parking functions (Theorem $14$ of \cite{DH17}). B\'ona conjectured~\cite{Bona-conjecture} that the area of a uniform random parking functions has log-concave distribution. By Lemma \ref{lem: Discrete to continuous} it follows that for B\'ona's conjecture to be true, the limiting distribution, which is the Airy distribution has to be log-concave. The following theorem shows that this is indeed true. In fact, Mohan Ravichandran (personal communication) has proved B\'ona's conjecture for all $n$.

\begin{theorem}\label{thm: Airy distribution}
    Airy distribution is log-concave.
\end{theorem}

The trick of conditioning log-concave density to a convex set can be extended to traceless Gaussian $\beta$-ensembles (see Section $2$ of \cite{Matsumoto08}). If we consider quadratic $V$ in $\beta$-Coulomb gases and restrict the density to the convex set $\mathcal{S}=\{x\in\mathcal{W}_n:\sum\limits_{i=1}^nx_i=0\}$, we obtain log-concavity of density of traceless Gaussian $\beta$-ensembles. In particular, we obtain log-concavity of largest eigenvalue of traceless GUE. The largest eigenvalue of a $k\times k$ traceless GUE is also the limiting distribution of the length of a longest weakly increasing subsequence of a random word from an ordered $k$ letter alphabet \cite{TW01}. One can ask whether log-concavity holds for each finite $k$ and $n$ (see Subsection \ref{subsec: open questions}). Traceless GUE is related to several other random word statistics \cite{KJ00, HM15}.

\section{Proofs of Theorem \ref{thm: Poisson Plancherel log-concave} and Theorem \ref{thm:gen weight}}\label{sec: discrete setting}

\begin{proof}[Proof of Theorem \ref{thm:gen weight}]

We first show that $\PP_{n,w,Q}$ is discrete log-concave on $\Z^n$. Note that for any $x,y \in \overrightarrow{\Z}^n $, we have $\left\lfloor \frac{x+y}{2}\right\rfloor, \left\lceil \frac{x+y}{2}\right\rceil \in \overrightarrow{\Z}^n$.

By the assumption \eqref{eq:log conc org}, we have that (See Remark \ref{rmk: org alt eq})
\begin{align*}
    w_i(x_i)w_i(y_i)\leq w_i\left(\left\lfloor \frac{x_i+y_i}{2}\right\rfloor\right)w_i\left(\left\lceil \frac{x_i+y_i}{2}\right\rceil\right).
\end{align*}

Hence in order to prove that $\PP_{n,w,Q}$ is discrete log-concave on $\Z^n$, it suffices to prove that for any $1\leq i<j\leq n$,
\begin{align}\label{eq: final for main thm}
    Q_{i,j}(x_j-x_i)Q(y_j-y_i)\leq Q_{i,j}\left(\left\lfloor \frac{x_j+y_j}{2}\right\rfloor-\left\lfloor \frac{x_i+y_i}{2}\right\rfloor\right)Q_{i,j}\left(\left\lceil \frac{x_j+y_j}{2}\right\rceil-\left\lceil \frac{x_i+y_i}{2}\right\rceil\right).
\end{align}

\textbf{Case 1:}
 If both $x_i+y_i$ and$\ x_j+y_j$ are  odd or both are even, we have 
\begin{align}\label{eq: both E or O}
    \left(\left\lfloor \frac{x_j+y_j}{2}\right\rfloor-\left\lfloor \frac{x_i+y_i}{2}\right\rfloor\right),\ \left(\left\lceil \frac{x_j+y_j}{2}\right\rceil-\left\lceil \frac{x_i+y_i}{2}\right\rceil\right)=\frac{(y_j-y_i)+(x_j-x_i)}{2}
\end{align}
As $Q_{i,j}$ is log-concave, $Q_{i,j}(a)Q_{i,j}(b)\leq Q_{i,j}^2\left(\frac{a+b}{2}\right)$ and \eqref{eq: final for main thm} follows from \eqref{eq: both E or O}.

\textbf{Case 2:} Now suppose $x_i+y_i$ is odd and $x_j+y_j$ is even, then
\begin{align}\label{eq: one O one E 1}
        \left\lfloor \frac{x_j+y_j}{2}\right\rfloor-\left\lfloor \frac{x_i+y_i}{2}\right\rfloor=\frac{(y_j-y_i)+(x_j-x_i)}{2}+\frac{1}{2}\\\label{eq: one O one E 2}
        \left\lceil \frac{x_j+y_j}{2}\right\rceil-\left\lceil \frac{x_i+y_i}{2}\right\rceil=\frac{(y_j-y_i)+(x_j-x_i)}{2}-\frac{1}{2}
\end{align}

Note that for $s,t$ and $u$ satisfying $s\leq s+u\leq t-u\leq t$, by log-concavity of $Q$, we have $Q(s)Q(t)\leq Q(s+u)Q(t-u)$. The said inequality might fail if $s=t$ and $u>0$. For that to happen we need $x_j-x_i=y_j-y_i$. One can check that for such $x_i,x_j,y_i,y_j$ we always have that parity of $x_i+y_i$ and $x_j+y_j$ match. Thus for Case $2$, we never have that $x_j-x_i=y_j-y_i$. Thus we have $Q(s)Q(t)\leq Q(s+u)Q(t-u)$. Using this inequality with \eqref{eq: one O one E 1} and \eqref{eq: one O one E 2} implies \eqref{eq: final for main thm}. Same argument can be used for the case when $x_i+y_i$ is even and $x_j+y_j$ is odd. 

Hence we have proved \eqref{eq: final for main thm} and hence also that $\PP_{n,w,Q}$ is discrete log-concave on $\Z^n$. 

The log-concavity of $h_i$ follows immediately from Corollary~\ref{cor:HKSmarginals}.
\end{proof}

\begin{remark}\label{rmk: org alt eq}
Although we use the condition that $\forall i,j\in\NN$
\begin{align}\label{eq:log conc alt}
    w(i)w(j)\leq w\left(\left\lfloor \frac{i+j}{2}\right\rfloor\right)w\left(\left\lceil \frac{i+j}{2}\right\rceil\right)
\end{align}
in the proof of Theorem \ref{thm:gen weight}, note that \eqref{eq:log conc alt} and \eqref{eq:log conc org} (with the no internal zeros) are equivalent. To see this, \eqref{eq:log conc org} implies that if $i\leq i+k\leq j-k\leq j$ then
\begin{align*}
w(i)w(j)\leq w(i+k)w(j-k),
\end{align*}
which gives us \eqref{eq:log conc alt}. Now for the other direction, taking $i=k-1$ and $j=k+1$ in \eqref{eq:log conc alt} gives us \eqref{eq:log conc org}.
\end{remark}

\begin{remark}\label{rmk: proof extension}
    Theorem \ref{thm:gen weight} can be extended to functions $Q_{i,j}(h_i,h_j)$ satisfying 
    \begin{align*}
        Q_{i,j}(h_i,h_j)Q_{i,j}(g_i,g_j)\leq Q_{i,j}\left(\left\lfloor \frac{h_i+g_i}{2}\right\rfloor, \left\lfloor \frac{h_j+g_j}{2}\right\rfloor\right)Q_{i,j}\left(\left\lceil \frac{h_i+g_i}{2}\right\rceil, \left\lceil \frac{h_j+g_j}{2}\right\rceil\right).
    \end{align*}
\end{remark}

\bprf[Proof of Theorem~\ref{thm:schurmeasures}] As in the proof of Theorem \ref{thm:gen weight}, we shall use Result~\ref{thm: Klartag result}. Writing
\begin{align*}
    \mathbb P_{a,b}(\lambda_i=k) = \frac{1}{Z_{a,b}}\sum_{\lambda:\lambda_i=k}s_{\lambda}(a)s_{\lambda}(b)
\end{align*}
we see that the  log-concavity of the distribution of $\lambda_i$ follows from Result~\ref{thm: Klartag result} if we could show that
\begin{align}\label{eq:requiredineqonschur}
    s_{\theta}(a)s_{\theta}(b)s_{\phi}(a)s_{\phi}(b)\le s_{\lambda}(a)s_{\lambda}(b)s_{\mu}(a)s_{\mu}(b)
\end{align}
where $\theta=\lfloor \frac{\lambda+\mu}{2}\rfloor$ and $\phi=\lceil \frac{\lambda+\mu}{2}\rceil$. Extending a conjecture of Okounkov \cite{okounkov1997}, it was proved by Lam, Postnikov and Pylyavskyy~\cite{lampostnikovpylyavskyy} that for  $\lambda,\mu,\theta,\phi$ related as above, 
\begin{align*}
s_{\theta}s_{\phi}\preceq    s_{\lambda}s_{\mu}
\end{align*}
where the inequality is in the sense of Schur positivity. That is, when  
$s_{\lambda}s_{\mu}-s_{\theta}s_{\phi}$ is expanded as a linear combination of Schur polynomials, the coefficients are all non-negative. Log-concavity of Schur polynomials has been used recently (see Section $4.4$ of \cite{DLM23} and Section $1.1$ of \cite{DLM24}) as a key ingredient in large deviation results. 

When a Schur polynomial is evaluated at $x=(x_1,x_2,\ldots )$ with $x_i\ge 0$, the result is non-negative (as clear from the definition $s_{\lambda}(x)=\sum_{T}x^T$, where the sum is over semistandard Young Tableaux $T$ of shape $\lambda$). Therefore, if $a_i\ge 0$ and $b_i\ge 0$, then 
\begin{align*}
s_{\theta}(a)s_{\phi}(a)\le    s_{\lambda}(a)s_{\mu}(a) \qquad \mbox{ and } \qquad s_{\theta}(b)s_{\phi}(b)\le    s_{\lambda}(b)s_{\mu}(b).
\end{align*}
Clearly \eqref{eq:requiredineqonschur} follows from this and the proof is complete.
\eprf

We now proceed with the proof of Theorem \ref{thm: Poisson Plancherel log-concave}.\\

There is a natural bijection from $h=(h_1,h_2,\dots,h_n)$ with $0\leq h_1< h_2<\dots<h_n$ to $\lambda$ with $\ell(\lambda)\leq n$, which is $\lambda_i=h_{n+1-i}-(n-i)$. Consider the discrete measure in \eqref{eq: DOPE} on $\overrightarrow{\NN}^n$ with $Q_{i,,j}(x)=Q(x)=x^{\beta}$ and $w_i(x)=w(x)=q^x$, where $0<q<1$. By the above bijection, such a measure on $\overrightarrow{\NN}^n$ induces a probability measure on $\Lambda$, say ${\gamma_{n,q,\beta}}$. 

\begin{theorem}\label{thm: Meixner limit to Plancherel}
   For $\alpha,\beta>0$, we have $\gamma_{n,\alpha/n^{\beta},\beta}$ converges in distribution to $M^{(\alpha,\beta)}$, as $n\rightarrow\infty$.
\end{theorem}

  Note that for $\beta=2$, Theorem \ref{thm: Meixner limit to Plancherel} is exactly the result, due to Johansson, that the limit of Meixner ensemble is Poissonized Plancherel measure. See Theorem $1.1$ of \cite{KJ02}. By Theorem \ref{thm:gen weight}, we have that $\forall i\in\N$, the distribution of $\lambda_i$ under the probability measure $\gamma_{n,q,\beta}$ is log-concave. Using Theorem \ref{thm: Meixner limit to Plancherel}, Theorem \ref{thm: Poisson Plancherel log-concave} is immediate.

In the proof of Theorem \ref{thm: Meixner limit to Plancherel}, we make use of the following formula, due to Frobenius determinant formula, for $d_{\lambda}$. If $\lambda\vdash k=(\lambda_1,\lambda_2,\dots,\lambda_\ell)$, then
\begin{align}\label{eq: Frobenius det formula}
    d_{\lambda}=\frac{k!\Delta(\lambda_{\ell},\lambda_{\ell-1}+1,\dots,\lambda_{1}+\ell-1)}{\lambda_{\ell}!(\lambda_{\ell-1}+1)!\dots(\lambda_{1}+\ell-1)!}.
\end{align}

\begin{proof}[Proof of Theorem \ref{thm: Meixner limit to Plancherel}]

 We will first show that, as $n\rightarrow\infty$, we have convergence of $R_{n,k}^{(\beta)}$ (as defined after \eqref{eq: rho definition}) to $\mu_{k}^{(\beta)}$. We then show that as $n\rightarrow\infty$, 
\begin{align}\label{eq: limiting mixture dist}
    \frac{\gamma_{n,\alpha/n^{\beta},\beta}(\sum \lambda_i=k+1)}{\gamma_{n,\alpha/n^{\beta},\beta}(\sum \lambda_i=k)}{\longrightarrow}\alpha\frac{\sum\limits_{\lambda\vdash k+1}(d_{\lambda}/(k+1)!)^{\beta}}{\sum\limits_{\lambda\vdash k}(d_{\lambda}/k!)^{\beta}}.
\end{align}

Note that to prove Theorem \ref{thm: Meixner limit to Plancherel}, it suffices to prove the above two claims. We now show that $R_{n,k}^{(\beta)}$ converges to $\mu_k^{(\beta)}$.

If $\lambda\vdash k=(\lambda_1,\lambda_2,\dots,\lambda_{\ell})$ which is mapped to $(h_1,h_2,\dots,h_n)$, one can check that
\begin{align}\label{eq: Plancherel Det formula for lambda}
    \prod\limits_{1\leq i<j\leq n}(h_j-h_i)^{\beta}= \prod\limits_{1\leq i<j\leq n-k}(h_j-h_i)^{\beta} (h_n!h_{n-1}!\dots h_{n-k+1}!)^{\beta}(d_{\lambda}/k!)^{\beta}.
\end{align}

Let $\lambda\vdash k$ and $\widehat{\lambda}\vdash k$ be two different partitions which are mapped to $h,\widehat{h}\in\overrightarrow{\NN}^n$. Note that this implies $\sum\limits_{i=n-k+1}^{n}h_i=\sum\limits_{i=n-k+1}^{n}\widehat{h}_i$. Then as $n\rightarrow\infty$,
\begin{align}\label{eq: limit to 1}
    \frac{\prod\limits_{1\leq i<j\leq n-k}(h_j-h_i)^{\beta} (h_n!h_{n-1}!\dots h_{n-k+1}!)^{\beta}}{\prod\limits_{1\leq i<j\leq n-k}(\widehat{h}_j-\widehat{h}_i)^{\beta} (\widehat{h}_n!\widehat{h}_{n-1}!\dots \widehat{h}_{n-k+1}!)^{\beta}}\rightarrow 1.
\end{align}
\eqref{eq: Frobenius det formula}, \eqref{eq: Plancherel Det formula for lambda}, \eqref{eq: limit to 1} together imply that $R_{n,k}^{(\beta)}$ converges to $\mu_k^{(\beta)}$. Now we prove \eqref{eq: limiting mixture dist}. 
\begin{align*}
    \gamma_{n,\alpha/n^{\beta},\beta}\left(\sum \lambda_i=k+1\right)\sim \sum\limits_{ \sum_{i}h_i=k+1+\frac{n(n-1)}{2}}\prod\limits_{i<j}(h_j-h_i)^{\beta}\left(\frac{\alpha}{n^{\beta}}\right)^{k+1+\frac{n(n-1)}{2}}.
\end{align*}

\begin{align}\label{eq: Plancherel limit}
    \lim_{n\rightarrow\infty}\frac{\gamma_{n,\alpha/n^{\beta},\beta}\left(\sum \lambda_i=k+1\right)}{\gamma_{n,\alpha/n^{\beta},\beta}\left(\sum \lambda_i=k\right)}=\lim_{n\rightarrow\infty}\frac{\alpha}{n^{\beta}}\frac{\sum\prod\limits_{i<j}(h_j-h_i)^{\beta}\ind_{\sum_{i}h_i=k+1+\frac{n(n-1)}{2}}}{\sum\prod\limits_{i<j}(h_j-h_i)^{\beta}\ind_{\sum_{i}h_i=k+\frac{n(n-1)}{2}}}.
\end{align}

Now we use \eqref{eq: Plancherel Det formula for lambda} to alternatively write each summand in both numerator and denominator of limit on the RHS of \eqref{eq: Plancherel limit}. Using Stirling's approximation it is a straight forward computation to check that \eqref{eq: limiting mixture dist} is true. This completes the proof of Theorem \ref{thm: Meixner limit to Plancherel}.    
\end{proof}

\section{Proofs of Theorem \ref{thm: Chen conjecture partial proof} and Theorem \ref{thm: Poissonization}}

\begin{proof}[Proof of Lemma \ref{lem: Discrete to continuous}]
    Suppose that, for the sake of contradiction, $f$ is not log-concave. Then there exists $x,y\in\R$ such that $f(x)f(y)>f^2(\frac{x+y}{2})$. Let $\mu_f$ be the probability measure corresponding to the density function $f$. Then $\frac{\mu_f(x-\varepsilon,x+\varepsilon)}{2\varepsilon}\rightarrow f(x)$. Choose $\varepsilon$ small enough so that,
    \begin{align*}
        {\mu_f(x-\varepsilon,x+\varepsilon)}{\mu_f(y-\varepsilon,y+\varepsilon)}>\left({\mu_f\left(\frac{x+y}{2}-\varepsilon-\varepsilon^2,\frac{x+y}{2}+\varepsilon+\varepsilon^2\right)}\right)^2.
    \end{align*}

As $\PP\left(\frac{X_n-a_n}{b_n}\in(x-\varepsilon,x+\varepsilon)\right)\rightarrow\mu_f(x-\varepsilon,x+\varepsilon)$, applying  Result \ref{thm: Klartag result} as 1-D discrete Brunn-Minkowski inequality, gives us the contradiction. Hence $f$ is log-concave.    
\end{proof}

\begin{proof}[Proof of Theorem \ref{thm: Chen conjecture partial proof}]
    Fix $j\in\NN$. We have that $\mu_n^{(2)}(\lambda)=\frac{d_{\lambda}^2}{n!}$. It is a simple calculation to check that, using \eqref{eq: Frobenius det formula}, for $k\in\{n-j,\dots,n\}$,

    \begin{align*}
       \lim_{n\rightarrow\infty} \frac{\mu_n^{(2)}(\lambda_1=k-1)\mu_n^{(2)}(\lambda_1=k+1)}{(\mu_n^{(2)}(\lambda_1=k))^2}=&\lim_{n\rightarrow\infty}\frac{\left(\sum\limits_{\lambda\vdash n-(k-1)}d_{\lambda}^2\right)\left(\sum\limits_{\lambda\vdash n-(k+1)}d_{\lambda}^2\right)}{\left(\sum\limits_{\lambda\vdash n-k}d_{\lambda}^2\right)^2}\\ 
       &\times \frac{(n-k)!^4}{(n-k-1)!^2(n-k+1)!^2}
    \end{align*}
This implies \eqref{eq: Chen conj result}.    
\end{proof}

\begin{proof}[Proof of Theorem \ref{thm: Poissonization}]
    In order to prove log-concavity of $Y$, we have to prove for any $k\geq 2$,
    \begin{align}\label{eq: poisson ineq}
       \left(\sum\limits_{i\geq 0}e^{-\lambda}\frac{\lambda^i}{i!}\mu_{i}(k-1)\right)
       \left(\sum\limits_{i\geq 0}e^{-\lambda}\frac{\lambda^i}{i!}\mu_{i}(k+1)\right)\leq 
       \left(\sum\limits_{i\geq 0}e^{-\lambda}\frac{\lambda^i}{i!}\mu_{i}(k)\right)^2. 
    \end{align}

    We define the functions $f,g,h=k$ as we did in the proof of Theorem \ref{thm:gen weight}.
    \begin{align*}
        h(x)=k(x)= e^{-\lambda}\frac{\lambda^{x}}{x!}\mu_{x}(k)\\ 
        f(x)= e^{-\lambda}\frac{\lambda^{x}}{x!}\mu_{x}(k-1)\\ 
        g(x)= e^{-\lambda}\frac{\lambda^{x}}{x!}\mu_{x}(k+1)
    \end{align*}

Using assumption \eqref{eq: poisson condition} we have that for any $i,j\geq 0$
\begin{align*}
    f(i)g(j)\leq h\left(\left\lfloor\frac{i+j}{2}\right\rfloor\right)k\left(\left\lceil\frac{i+j}{2}\right\rceil\right).
\end{align*}

This verifies the condition \eqref{eq:Klartag conditions} for the above defined functions $f,g,h=k$ when $n=1$. Applying Result \ref{thm: Klartag result}, we get that \eqref{eq: poisson ineq} is true. This completes the proof of log-concavity of $Y$.
\end{proof}

\section{Proofs of Theorem \ref{thm: Airy_2 process} and Theorem \ref{thm: Airy distribution}}
\begin{proof}[Proof of Theorem \ref{thm: Airy_2 process}]
    We use the fact that for any $N$, we can obtain $\left( B_1(t),\dots, B_N(t)\right)$ by conditioning a collection of $N$ independent Brownian bridges sequentially. Let $\left( W_1(t),\dots, W_N(t)\right)$ be a collection of independent Brownian bridges with all starting and ending at zero at times $0$ and $1$ respectively. For any $t_{j,1}<\dots<t_{j,j}$, the joint distribution 
\begin{align*}
\left( W_1(t_{j,1}),\dots, W_N(t_{j,1}),W_1(t_{j,2}),\dots, W_N(t_{j,2}),W_1(t_{j,j}),\dots, W_N(t_{j,j})\right)    
\end{align*}
is log-concave as it is a Gaussian vector. Now conditioning on the event 
\begin{align*}
    E_j= \{W_1(t_{j,i})<\dots<W_N(t_{j,i}),\ \forall i\in[j]\},
\end{align*}
is just restricting the Gaussian density to the convex set, 
\begin{align*}
\{x\in \mathbb{R}^{jN}: x_{i,N+1}<\dots<x_{i,N+n}, \forall i\in\{0,1,\dots,j-1\}\}\end{align*}
 on which log-concavity of the joint distribution would still hold. Hence conditional on $E_j$, the joint distribution $(W_N(t_{j,1}),\dots,W_N(t_{j,j}))$ is log-concave (Pr\'ekopa-Leindler inequality). Note that  
 \begin{align*}
     \left( W_1(t),\dots, W_N(t)\right)\text{conditioned on $E_j$} \rightarrow \left( B_1(t),\dots, B_N(t)\right) \text{conditioned on non-intersection}
 \end{align*}
  with the mesh $t_{j,1}<\dots<t_{j,j}$ converging to $(0,1)$ as $j\rightarrow\infty$. Also for any given $t_1<\dots<t_k$, one can choose a mesh converging to $(0,1)$ which contain $t_1,\dots,t_k$ at all times. Using Pr\'ekopa-Leindler inequality on the appropriate marginals, we obtain that $(B_N(t_1),\dots,B_N(t_k))$ is log-concave. By \eqref{eq: BM to Airy process} and preservation of log-concavity under translation, we have that \newline
  $\left(\mathcal{A}_2(t_1),\dots, \mathcal{A}_2(t_k)\right)$ is log-concave.
\end{proof}

\begin{proof}[Proof of Theorem \ref{thm: Airy distribution}]

Let $\{X_t\}_{t\in[0,1]}$ be a Brownian bridge. For each $n\in\N$, the joint distribution \newline 
$(X_{1/2^n}, X_{2/2^n}, \dots, X_{1 - 1/2^n})$ has log-concave density, as $X_t$ is a Gaussian process. Let $X_t^{(2^n)}$ be the process after conditioning on the event
\begin{align}
    S_{2^n}=\left\{\min\limits_{k\in \{1/2^n, 2/2^n,\dots, 1-1/2^n\}}X_k>0\right\}.
\end{align}

As restriction of log-concave density to a convex set is log-concave, the joint distribution $(X_{1/2^n}^{(2^n)}$, $ X_{2/2^n}^{(2^n)},$ $ \dots, X_{1 - 1/2^n}^{(2^n)})$ has log-concave density. As the class of log-concave random vectors is closed under linear transformations, using Pr\'ekopa-Leindler inequality, for any $A\subset[2^n-1]$, we have that   $\sum\limits_{k\in A}X_{k/2^n}^{(2^n)}$ is log-concave random variable. As $X_t^{(2^n)}$ converges weakly to $B_t$, for any $m\in\N$, $\sum\limits_{k=1}^{2^m-1}X_{k/2^m}^{(2^n)}/2^m$ converges to $\sum\limits_{k=1}^{2^m-1}B^{ex}({k/2^m})/2^m$ weakly as $n\rightarrow\infty$. This implies $\sum\limits_{k=1}^{2^m-1}B^{ex}({k/2^m})/2^m$ is log-concave. By letting $m\rightarrow\infty$, we have that $A$ is log-concave random variable.
\end{proof}

\section{Additional remarks and open questions}\label{subsec: open questions}

For the rest of the section, we discuss a few open questions extending the results mentioned above for various ensembles.  

\vspace{0.5cm}
\textbf{Open questions:}
\begin{enumerate}[(i)]
    \item Let $\rho_{n,k}^{(\beta)}$ be the probability measure on $h\in\overrightarrow{\NN}^n$ defined such that,
    \begin{align}\label{eq: rho definition}
    \rho_{n,k}^{(\beta)}(h=(h_1<h_2<\dots<h_n))\sim\prod\limits_{1\leq i<j\leq n}(h_j-h_i)^{\beta}\ind_{\sum h_i=k+\frac{n(n-1)}{2}}.
    \end{align}

    $\rho_{n,k}^{(\beta)}$ induces a probability measure on $\Lambda_k$, say $R_{n,k}^{(\beta)}$, due to the natural bijection for $n>k$. We explain this bijection for $n=4$ and $k=3$. For $\sum h_i=k+\frac{n(n-1)}{2}$, we need to move some $h_i$s to right from their initial locations at $i-1$. Suppose $0,1,3,5$ are the locations of $h_i$s, then $h_3,h_4$ were moved $1$ and $2$ places to the right of their initial locations. We hence map it to the partition $\lambda=(2,1)$. 

    Note that $\rho_{n,k}^{(2)}$ is exactly $\mathbb{P}_{n,n,\mbox{Me}}$ conditioned on $\sum h_i=k+\frac{n(n-1)}{2}$. It can be shown that $R_{n,k}^{(2)}$ converges to $\mu_k^{(2)}$ as $n\rightarrow\infty$ (see first claim in the proof of Theorem \ref{thm: Meixner limit to Plancherel}). Thus \eqref{eq: RSK chen conjecture} follows, which is equivalent to Conjecture \ref{conj: Chen conjecture}, if for $n>k$,
    \begin{align}\label{conj: stronger chen conjecture}
        \rho_{n,k}^{(2)}(h_n=j-1)\rho_{n,k}^{(2)}(h_n=j+1)\leq \rho_{n,k}^{(2)}(h_n=j)^2
    \end{align}
    holds. Note that \eqref{conj: stronger chen conjecture} is a generalization of Chen's conjecture and is checked to be true for small $n,k$.
    \vspace{0.2cm}
    \item  Also given that Theorem \ref{thm: Poisson Plancherel log-concave} holds for all $\lambda_i$, it would be interesting to know whether the distribution of $\lambda_2, \lambda_3,\dots$ are also log-concave under the Plancherel measure $\mu_n^{(2)}.$ It would also be interesting to know if the distribution of the sum of first few rows is log-concave.
   \vspace{0.05cm}    
    \item Another combinatorial object related to discrete ensembles is random words. Denote $\ell_{m,n}$ to be the length of longest weakly increasing subsequence of a word of length $n$ chosen uniformly random from ordered alphabet $\{1,2,\dots,m\}$. It is known that if $n\sim \mbox{Poi}(\alpha)$ then $\ell_{m,n}$ has the same distribution as $\mathbb{P}_{m,\alpha,\mbox{Ch}}(h_m)$ up to a shift (Proposition $1.5$ of \cite{KJ02}). Hence under Poissonization the distribution of $\ell_{w,m,n}$ is log-concave. Also $\ell_{w,m,n}$ is also distributed as $h_m$ with $\mathbb{P}_{m,\alpha,\mbox{Ch}}$ conditioned on $\sum h_i=n+\frac{m(m-1)}{2}$. Thus as before one could consider whether for fixed $m$ and $n$ the below inequality holds for all $i$, 
    \begin{align}\label{eq: random word chen conjecture}
        \PP(\ell_{m,n}=i-1)\PP(\ell_{m,n}=i+1)\leq \PP(\ell_{m,n}=i)^2.
    \end{align}
    Note that \eqref{eq: random word chen conjecture} is a random word variant of Chen's conjecture and is checked to be true for $1\leq m,n\leq 10$.
    \vspace{0.1cm}
    
        \item Similar questions could be asked for Krawtchouk ensemble, which is related to zig-zag paths in random domino tilings of Aztec diamond (see \cite{KJ02, EKLP}
    ) and for Hahn ensembles, which is related to random tilings of a hexagon (see  \cite{KJ02, CLP}).
    \vspace{0.1cm}
    \item  A problem similar to longest increasing subsequence, but of which very little is known is the length of longest common subsequence between two random words  of ordered alphabet which are of same length. Similar to Conjecture \ref{conj: Chen conjecture}, we could also ask whether length of longest common subsequence has log-concave distribution. Our simulations, for binary words show that this is indeed true for small $n$. One could also consider similar question for length of common subsequence between pairs of random permutations of $[n]$. The limiting distribution of fluctuations is known to be $TW_2$ \cite{HI23}. 
    \vspace{0.1cm}
    \item As remarked earlier, the log-concavity of exponential last passage time follows can be shown using Theorem \ref{thm:schurmeasures}. Consider the location of final point in the point to line passage time, which is the obtained from taking geometric limit to exponentials in $G_{\mathbf{1},\mathbf{n}}^{(1)}$. Although our methods cannot prove it, from simulations it is found that the location of this final point also has log-concave distribution on the line $x+y=2n$. It would be interesting to know if this is true.  It would also be interesting to know if log-concavity of last passage times could be proven for by some other general method which would also work for models which do not fall in to integrable systems (weights other than geometric and exponential).
    \vspace{0.1cm}

    \item We finally consider $TW_{\beta}$ distributions. For a positive integer $r$, a measurable function $f:\R\rightarrow\R$ is called P\'olya frequency function of order $r$, written as $\mbox{PF}_r$, if $\det \l(p(x_i-y_j)\r)_{i,j=1}^{m}\geq 0$ for all choices of $x_1<x_2\dots<x_m$ and $y_1<y_2\dots<y_m$ for all $1\leq m\leq r$ (the matrix $[p(x_i-y_j)]_{1\leq i,j\leq n}$ is totally positive). A function is $\mbox{PF}_2$ if and only if $f$ is log-concave (see \cite{SW14}). Thus by Corollary \ref{cor: Stoch Airy operator} we have that $TW_{\beta}$ densities are $\mbox{PF}_2$. $PF_{\infty}$ probability density functions (functions which are $PF_{r}$ for all $r\geq 0$) can be characterised as density functions of a linear combination of independent exponentials up to an independent Gaussian difference (see Theorem $2.4$ of \cite{BGKP22}). It follows easily that such measures have $\PP\l(X\geq t\r)\geq \exp(-ct)$ for some $c>0$ and all large $t$. But the tails of $TW_{\beta}$ are of the order $\exp(-c_{\beta}t^{3/2})$ \cite{RRV06}. Hence it follows that $TW_{\beta}$ cannot be $PF_{\infty}$. It is a natural question as to what is the largest $r$ such that $TW_{\beta}$ are $PF_{r}$? 

    \item One can also consider whether the following holds for all $\beta\geq 1$, for all $k\leq n-1$ and for large enough $n$.
    \begin{align}\label{eq: RSK beta chen conjecture}
\mu_n^{(\beta)}(\lambda_1=k-1)\mu_n^{(\beta)}(\lambda_1=k+1)\leq (\mu_n^{(\beta)}(\lambda_1=k))^2.    
\end{align}
Note that for $\beta=2$ this is Conjecture \ref{conj: Chen conjecture}. The corresponding inequality for $\beta=1$ is equivalent to the B\'ona-Lackner-Sagan conjecture on involutions \cite[Conjecture $1.2$]{BLS17}. We observe here that the above inequality does not hold for $\beta<1$ for large enough $n$. Using \eqref{eq: beta plancherel measure} and hook length formula to compute $d_{\lambda},$ we see that 
\begin{align*}
    \frac{\mu_n^{(\beta)}(\lambda_1=n-2)\mu_n^{(\beta)}(\lambda_1=n)}{\l(\mu_n^{(\beta)}(\lambda_1=n-1)\r)^2}=\frac{(\frac{n-3}{2})^{\beta}+(\frac{n-2}{2})^{\beta}}{(n-1)^{\beta}}.
\end{align*}   
For any $\beta<1$, the above ratio is greater than $1$ for large enough $n$. Hence \eqref{eq: RSK beta chen conjecture} fails for any $\beta<1$, for large enough $n$ . 
\end{enumerate}

\textbf{Acknowledgements: }  The authors would like to thank Mohan Ravichandran for raising the question of log-concavity of Airy distribution and explaining its occurrence in the study of random parking functions. The authors would also like to thank Milind Hegde for pointing out Okounkov's conjecture and its connection to log-concavity of Schur measures. The authors would also like to thank Joseph Lehec, Paul-Marie Samson, Dylan Langharst, Pietro Caputo, Cyril Roberto, James Melbourne, Krzysztof Oleszkiewicz, Christian Houdr\'{e} and Emanuel Milman for helpful discussions.

\appendix

\section{From the Painlev\'{e} description to log-concavity of \texorpdfstring{$TW_2$ }  aDistribution }\label{appendix}

Here we provide an alternate proof of the result that $TW_2$ is log-concave.   We use the following description of cumulative  distribution function (c.d.f.) of $TW_2$ distribution. Let $F_2(x)$ be the c.d.f. of $TW_2$ distribution and $Ai(x)$ be the Airy function for $x\in \mathbb{R}$ given by
\begin{align*}
    Ai(x)=\frac{1}{\pi}\int_{0}^{\infty}\cos \left( \frac{t^3}{3}+xt\right) dt.
\end{align*}
It is standard result that $Ai(x)\sim \frac{1}{\sqrt{2\pi}z^{1/4}}\exp(-\frac{2}{3}x^{3/2}),$ as $ x\rightarrow\infty$.
\begin{theorem}[Theorem $3.1.5$, \cite{AGZ}]\label{TW cdf}
    The function $F_2(x)$ admits the representation
    \begin{align}\label{cdf}
        F_2(x)=\exp\left( -\int_{x}^{\infty}(t-x)\ u^2(t) dt\right),
    \end{align}
    where $u$ satisfies 
    \begin{align}\label{Painleve}
        u''(x)=xu(x)+2u^3(x),
    \end{align}
    with $ u(x)\sim Ai(x), \ \mbox{as }x\rightarrow +\infty.$
\end{theorem}

Equation \eqref{Painleve} is the Painlev\'e equation of type II. Many properties of the solutions of \eqref{Painleve} are deferred to later. Note that a twice differentiable function $f:\mathbb{R}\rightarrow \mathbb{R}$ is log-concave on $\mathbb{R}$, if $(\log f)''(x)\leq 0, \forall x\in \mathbb{R}$.

    First we prove a lemma which shows that if the function $u$ in \eqref{cdf} does not have any zeros, then density of $TW_2$ distribution is log-concave on $\mathbb{R}$. We then show that indeed the solution $u(x)$ has no zeros. For the rest of the article we denote $F_2(x)$ as $F(x)$.

    \begin{lemma}\label{lemma}
        If $u(x)$ is a solution of \eqref{Painleve} and $u(x)\sim Ai(x), $ as $x\rightarrow +\infty$, then $(\log F'(x))''\leq 0, \ \forall x\in \mathbb{R}$.
    \end{lemma}

\begin{proof}[Proof of Lemma \ref{lemma}] Define \begin{align*}
    h(x)=\int_{x}^{\infty}u^2(t) dt.
\end{align*} 

We make a note of the following functions. 
\begin{align*}
h'(x)&=-u^2(x)\\
    F(x)&=\exp\left( -\int_{x}^{\infty}(t-x)\ u^2(t) dt\right)\\
    F'(x)&=F(x) h(x)\\
    F''(x)&=F'(x)h(x)+F(x)h'(x)\\
    &=F(x)(h^2(x)-u^2(x))\\
    F'''(x)&=F'(x)(h^2-u^2)+F(x)(2hh'-2uu')\\
    &= F(x)(h^3-3u^2h-2uu')\\
    (\log F'(x))'&=\frac{F''(x)}{F'(x)}\\
    (\log F'(x))''&=\frac{F''' F'-(F'')^2}{(F')^2}\\
    &=\frac{-u^4-u^2h^2-2uu'h}{h^2}.
\end{align*}
    As we want to show $(\log F'(x))''\leq 0$, it is enough to show that 
    \begin{align}\label{non neg}
        u^4+u^2h^2+2uu'h\geq 0.
    \end{align}
 Dividing \eqref{non neg} by $u^2(x)$, it is enough to show
 \begin{align}\label{g}
     g(x)= u^2+h^2+2h\frac{u'}{u}\geq 0.
 \end{align}
Here we have used the assumption that $u$ has no zeros, which makes the function $g(x)$ well defined. We will show that $g(x)\rightarrow 0,$ as $x\rightarrow +\infty$ and that $g'(x)\leq 0$. This implies $g(x)\geq 0$, $\forall x\in \mathbb{R}$.
\begin{align}\label{g'}
    g'(x)=-2h \frac{u^4-u''u+(u')^2}{u^2}.
\end{align}

Multiplying \eqref{Painleve} by $u'(x)$ and integrating $x$ to $\infty$, we get that, using boundary conditions, 
\begin{align}\label{First derivative}
    (u'(x))^2=xu^2(x)+h(x)+u^4(x).
\end{align}
Using \eqref{First derivative} and \eqref{Painleve} in \eqref{g'}, we get that $g'(x)=-2\frac{h^2}{u^2}<0$. We now show that $g(x)\rightarrow 0$.

Although it is shown in the proof of Theorem $5.1$ of \cite{BLS17} that $g(x)\rightarrow 0$, we give a slightly different argument.
Define $v(x)=-{u'(x)}/{u(x)}.$ By \eqref{First derivative}, 
\begin{align}\label{v^2}
    v^2(x)=x+\frac{h(x)}{u^2(x)}+u^2(x).
\end{align}

Using standard asymptotics of $Ai(x), Ai'(x)$, we have that,
\begin{align*}
    \frac{Ai(x)}{Ai'(x)}\sim -1/\sqrt{x}, \quad x\rightarrow \infty.
\end{align*}
Applying l'H\^opital's rule to $\frac{h}{u^2}$ and using the fact that $u(x)\sim Ai(x)$, \eqref{v^2} gives $v(x)\sim \sqrt{x}$. As it is known that $h(x)$ decreases as $\exp(-x^{3/2})$ we get $h(x)v(x)\rightarrow 0$. This gives that $g(x)$ in \eqref{g} goes to $0$, as $x\rightarrow\infty$. This completes the proof of the lemma.
\end{proof}

Now we shall show that the solution to \eqref{Painleve} satisfying the boundary condition $u(x)\sim Ai(x),$ as $ x\rightarrow\infty$, has no zeros. In fact we show that $u(x)$ is monotonically decreasing and since $u(x)\sim Ai(x)$ we have $u(x)> 0$.

As we could not find a quotable reference stating that $u(x)$ is monotonically decreasing, we state the result in the form of a lemma. Note that existence and uniqueness of solution to \eqref{Painleve} has been proven in \cite{HM80}.

\begin{lemma}\label{neg asymp}
    If $u(x)$ is a solution to \eqref{Painleve} and $u(x)\sim Ai(x)$ as $x\rightarrow\infty$, then $u(x)$ is a non-increasing function with $u(x)\sim \sqrt{\frac{-x}{2}}$ as $x\rightarrow -\infty$.
\end{lemma}

\begin{proof}[Proof of Lemma \ref{neg asymp}]
    We use the following results about $u(x)$ from Theorem $1$ and Theorem $2$ of \cite{HM80}. 
    
    If $u(x)$ is a solution of \eqref{Painleve} and $u(x)\rightarrow 0$ as $x\rightarrow\infty$ and $u(x)\sim \sqrt{\frac{-x}{2}}$ as $x\rightarrow -\infty$,
    \begin{itemize}
        \item $u(x)$ is a unique solution satisfying $u(x)\sim Ai(x)$ as $x\rightarrow \infty$.
        \item $u(x)>0, u'(x)<0 $ for $x\geq 0$.
        \item $u''(x)$ has exactly one zero.
        \item $u''(x)<0$ for large negative $x$ and $u''(x)>0$ for large positive $x$.
    \end{itemize}
So by the assumptions of the lemma, we have $u(x)\sim \sqrt{\frac{-x}{2}}$ as $x\rightarrow -\infty$. We are left to show $u'(x)\leq 0, \forall x\in \mathbb{R}$.

Suppose $u'(x_0)>0$ for some $x_0$. As $u'(x)<0$ for $x>0$, there must be some $x_1>x_0,$ such that $ u'(x_1)=0$ and $ u''(x_1)<0$ ($x_1$ is a local maxima). As $u(x)\sim \sqrt{\frac{-x}{2}}$ as $x\rightarrow -\infty$, there must also be some $x_2<x_0$ such that $u'(x_2)=0$ and  $u''(x_2)>0$ ($x_2$ is a local minima).

As $u''(x)>0$ for large positive $x$ and $u''(x)<0$ for large negative $x$, there must exist $x_3>x_1$ such that $u''(x_3)=0$ and there must also exist $x_4<x_2$ such that $u''(x_4)=0$. This would mean $u''(x)$ has tow distinct zeros which contradicts the earlier result that $u''(x)$ has only one zero. Hence $u'(x)\leq 0$. This implies that $u(x)$ is non increasing. This completes the proof of Lemma \ref{neg asymp}.    
\end{proof}

Lemma \ref{lemma} and Lemma \ref{neg asymp} together imply that $TW_2$ is log-concave.  

\bigskip





\bibliographystyle{abbrv}

\bibliography{pustak1.bib}

\end{document}